\documentclass[8pt,a4paper]{article}

\usepackage{amsmath,amsfonts,amssymb}
\usepackage{bbm}
\usepackage{cases}

\usepackage{kotex}
\usepackage{lipsum}
\usepackage{subcaption}
\usepackage{xcolor}
\usepackage{graphicx}
\usepackage{ulem}
\newtheorem{defn}{Definition}[section]
\newtheorem{theo}[defn]{Theorem}
\newtheorem{lem}[defn]{Lemma}
\newtheorem{prop}[defn]{Proposition}
\newtheorem{cor}[defn]{Corollary}
\newtheorem{rem}[defn]{Remark}

\newenvironment{proof}{{\bf Proof }}{{\vskip 0.1cm \hfill$\Box$}}
\begin{document} 

\noindent
{\Large \bf A weighted semigroup approach to exponential stability in linear parabolic equations}
\\ \\
\bigskip
\noindent
{\bf Haesung Lee}  \\
\noindent
{\bf Abstract.}  
This paper establishes the exponential $L^2$-stability of the unique solutions to initial-boundary value problems for linear parabolic partial differential equations with general drift and zero-order coefficients in bounded domains. The key idea lies in constructing a suitable Dirichlet form with respect to a weighted measure $\mu = \rho\,dx$ and identifying the corresponding sub-Markovian $C_0$-semigroup of contractions on $L^2(U, \mu)$ with the unique weak solution. Remarkably, the exponential $L^2$-stability remains valid even when the zero-order term vanishes, and it holds robustly for all drift coefficients $\mathbf{H} \in L^p(U, \mathbb{R}^d)$ with $p \in (d, \infty)$.
\\ \\
\noindent
{Mathematics Subject Classification (2020): {Primary: 35B40, 47D07, 31C25, Secondary: 35K20, 35K90, 60J60}}\\

\noindent 
{Keywords: Dirichlet forms, exponential stability, weighted semigroups, initial-boundary value problems, weak solutions, divergence-free transformation
}

\section{Introduction} \label{intro}
In this paper, we consider the initial-boundary value problem (Cauchy-Dirichlet problem) for a linear parabolic partial differential equation in a bounded domain $U \subset \mathbb{R}^d$ with $d \geq 2$:
\begin{equation} \label{bvpara}
\left\{
\begin{aligned}
    u_t - \operatorname{div}(A \nabla u) + \langle \mathbf{H}, \nabla u \rangle + (c + \theta)u &= 0 && \text{in } U \times (0, T), \\
    u &= 0 && \text{on } \partial U \times [0, T], \\
    u &= g && \text{on } U \times \{t = 0\},
\end{aligned}
\right.
\end{equation}
where the precise assumptions on the coefficients and initial data are stated in our main result, Theorem~\ref{maintheore}. The primary goal of this paper is to establish not only the existence and uniqueness of weak solutions to \eqref{bvpara}, but also their exponential stability toward zero in $L^2(U)$ (see Theorem~\ref{maintheore}).\\
The well-posedness of initial-boundary value problems for linear parabolic equations such as \eqref{bvpara} has been extensively studied in the mathematical literature.  For instance, in a standard PDE textbook \cite{E10}, the existence result for \eqref{bvpara} is obtained via Galerkin methods under the assumption that the coefficients are bounded (see \cite[Section 7.1]{E10}). If the coefficients are additionally Hölder continuous and the initial data $g$ is continuous on $\overline{U}$, then the existence and uniqueness of classical solutions are also guaranteed in \cite[Chapter 2 and Proposition C.3.2]{Lo07} (cf. \cite{MPW02}).\\
As in the elliptic case, the H\"{o}lder continuity condition on the coefficients appears to be essential for the existence of classical solutions.  However, in many applications, the coefficients may fail to be continuous, and in particular, the drift coefficient $\mathbf{H}$ and the zero-order coefficient $c+\theta$ may be locally unbounded and the initial data $g$ is given only as a function in $L^2(U)$. In such cases, it is natural to consider weak solutions in the sense of Definition \ref{startdefn}.\\
The existence of weak solutions to \eqref{bvpara} under appropriate coercivity assumptions on the associated bilinear form can be found in \cite{LM72} (cf. \cite[Theorem 10.9]{Br11}). In particular, when the coefficients are time-independent, semigroup theory becomes a powerful tool for analyzing a solution to \eqref{bvpara}. For example, \cite[Section~7.4]{E10} constructs a strongly continuous semigroup on $L^2(U)$ via the Hille-Yosida theorem and identifies the weak solution to \eqref{bvpara} under the assumption that $\mathbf{H}$ and $c+\theta$ are bounded (cf. \cite{RR04, Jo13}). Moreover, the constructed semigroup satisfies a $\gamma$-contraction property, i.e., there exists a constant $\gamma > 0$ such that
\begin{equation} \label{gammcontrac}
\|u(\cdot, t)\|_{L^2(U)} \leq e^{\gamma t} \|g\|_{L^2(U)}  \quad \text{ for all $t \in (0, \infty)$},
\end{equation}
as shown in \cite[Section~7.4, Theorem~5]{E10}. However, the estimate \eqref{gammcontrac} becomes ineffective for large $t$, since its right-hand side grows exponentially. Moreover, the constant $\gamma$ depends on the coefficients, and its explicit computation is generally not feasible in closed form. \\
In contrast, when the coefficients are H\"{o}lder continuous, $c \geq 0$ in $U$ and a constant $\theta > 0$ is assumed, one can derive an explicit exponential decay estimate of the form in $L^{\infty}(U)$:
\begin{equation} \label{expondecaestim}
\|u(\cdot, t)\|_{L^{\infty}(U)} \leq e^{-\theta t} \|g\|_{L^{\infty}(U)},
\end{equation}
as shown in \cite[Proposition~C.3.2]{Lo07}. \\ 
This motivates the investigation of the critical case $c \equiv 0$, $\theta = 0$, in which no zero-order term is present. The case where the zero-order term vanishes, that is, $c + \theta \equiv 0$, arises naturally in stochastic analysis (see \cite{L25, L25ai, L25jm}). In particular, the Kolmogorov equation associated with a diffusion process typically involves a second-order partial differential operator without a zero-order term. In such settings, no artificial damping parameter $\theta$ is introduced, yet the long-time behavior of solutions remains a central and subtle issue.
To the best of our knowledge, there is no existing general result that proves the exponential $L^2$-decay of solutions to \eqref{bvpara} in this critical setting,  $c + \theta \equiv 0$,  even when the coefficients $A$ and $\mathbf{H}$ are smooth.\\
In the case where \eqref{bvpara} corresponds to the heat equation, the spectral decomposition of a solution with respect to the Dirichlet Laplacian
\[
u(\cdot, t) = \sum_{n=1}^\infty e^{-\lambda_n t} \langle g, e_n\rangle_{L^2(U)} e_n,
\]
where $\{e_n\}_{n \geq 1}$ and $\{\lambda_n\}_{n \geq 1}$ denote the Dirichlet eigenfunctions and the corresponding eigenvalues of $-\Delta$, respectively, arranged in increasing order $0 < \lambda_1 \leq \lambda_2 \leq \cdots \to \infty$ (see \cite[Section 8.2]{Ar23}), 
immediately implies the exponential convergence
\[
\|u(\cdot, t)\|_{L^2(U)} \leq e^{-\lambda_1 t} \|g\|_{L^2(U)}.
\]
In particular, for the Laplacian with Dirichlet boundary conditions, the first eigenvalue $\lambda_1$ can be characterized variationally (see \cite[Section 6.5]{E10}). \\
To address more general operators, one can employ a spectral approach based on abstract semigroup theory. In particular, the exponential decay estimate
\begin{equation*}% \label{exponenestim}
\|u(\cdot, t)\|_{L^2(U)} \leq e^{-\lambda_1 t} \|g\|_{L^2(U)}, \quad \text{for all } t > 0,
\end{equation*}
can be established under suitable spectral and semigroup-growth assumptions on the generator (see \cite[Chapter 4]{P83} and \cite[Chapter 5]{ABHFN11}). However, when the domain is irregular or the operator is non-self-adjoint, particularly in the presence of nontrivial drift coefficients $\mathbf{H}$, the verification of such a spectral bound and its explicit computation  become significantly more difficult. \\
 In fact, even if the drift coefficient $\mathbf{H}$ is smooth, the presence of a strongly positive divergence, i.e., $\operatorname{div} \mathbf{H} \gg 0$, may destroy coercivity and lead one to suspect that exponential decay fails. Surprisingly, the main results of this paper show that such pessimistic expectations, namely that exponential decay may fail under strong drift, are not valid. Specifically, in this paper we prove that the unique solution to \eqref{bvpara} converges exponentially to zero under general conditions, even in the presence of highly nontrivial drift. Before stating our main result, we introduce the main structural assumptions {\bf (S)} on the coefficients $A$ and $\mathbf{H}$.
\begin{description}
\item[(S):] $U$ is a bounded open subset of $\mathbb{R}^d$ with $d \geq 2$, $B_R(x_0)$ is an open ball in $\mathbb{R}^d$ with $\overline{U} \subset B_R(x_0)$, $\mathbf{H} \in L^p(U, \mathbb{R}^d)$, $h \in L^p(U)$ with $p \in (d, \infty)$ satisfies $\|\mathbf{H}\| \leq h$ a.e. in $U$. We extend $\mathbf{H}$ and $h$ by zero to $\mathbb{R}^d$, and denote the extensions again by $\mathbf{H}$ and $h$, respectively.
$A = (a_{ij})_{1 \leq i,j \leq d}$ is a (possibly non-symmetric) matrix of measurable functions on $\mathbb{R}^d$ such that for some constants $M > 0$ and $\lambda > 0$, it holds that
\begin{equation} \label{ellipticity}
\max_{1 \leq i,j \leq d} |a_{ij}(x)| \leq M, \quad \langle A(x)\xi, \xi \rangle \geq \lambda \|\xi\|^2 \quad \text{for a.e. } x \in \mathbb{R}^d \text{ and for all } \xi \in \mathbb{R}^d.
\end{equation}
\end{description}
We now present our main result, which reveals a robust mechanism for exponential decay even in the absence of standard coercivity conditions. Notably, the $L^2$-exponential decay estimate in Theorem~\ref{maintheore}(ii) is obtained without assuming any symmetry of the underlying operator, indicating a certain robustness in the analytic framework that follows.
\begin{theo} \label{maintheore}
Assume that {\bf (S)} holds, $\theta \in [0, \infty)$ is a constant, and $c \in L^s(U)$ satisfies $c \geq 0$ a.e. in $U$, where $s \in (1, \infty)$ if $d = 2$ and $s := \frac{d}{2}$ if $d \geq 3$. Let $g \in L^2(U)$, $T \in (0, \infty)$, and let $\mathbf{u} \in C([0, \infty); L^2(U)) \cap L^2(0,T; H^{1,2}_0(U))$ denote the (unique) weak solution to \eqref{bvpara} constructed in Theorem~\ref{mainwellposth} (see Definition \ref{startdefn}). Then the following estimates hold:
\begin{itemize}
\item[(i)]
\begin{equation} \label{energyrhonnewori}
\|\mathbf{u} \|_{L^2(0,T;H^{1,2}_0(U))} \leq \left(\frac{K_1}{2\lambda} \right)^{\frac12} \| g\|_{L^2(U)},
\end{equation}
where $K_1 \geq 1$ is the constant from Theorem~\ref{existinvar}, depending only on $d, \lambda, M, R, p$, and $\|h\|_{L^p(U)}$.
\item[(ii)]
\begin{equation} \label{expdecaacer}
\|\mathbf{u}(t)\|_{L^2(U)} \leq K_1^{\frac12} e^{-\kappa t} \|g\|_{L^2(U)} \quad \text{for all } t \in (0, \infty),
\end{equation}
where 
$$
\kappa := \theta + \frac{d^2 \lambda}{4K_1 (d - 1)^2 |U|^{\frac{2}{d}}}.
$$
\item[(iii)]
Let $r \in [1, \infty)$. If $g \in L^2(U) \cap L^r(U)$, then $\mathbf{u} \in C([0, \infty); L^r(U))$, and
\begin{equation} \label{linftcontrarca}
\| \mathbf{u}(t) \|_{L^r(U)} \leq K_1^{\frac1r} e^{-\theta t} \|g \|_{L^r(U)} \quad \text{for all } t \in (0, \infty).
\end{equation}
Moreover, if $g \in L^{\infty}(U)$, then
\begin{equation} \label{linftycontra}
\| \mathbf{u}(t) \|_{L^{\infty}(U)} \leq e^{-\theta t} \|g \|_{L^{\infty}(U)} \quad \text{for all } t \in (0, \infty).
\end{equation}
\end{itemize}
\end{theo}
We now outline the main steps underlying the derivation of our main result. First, the existence and uniqueness of weak solutions to \eqref{bvpara} are established in Section~\ref{semigrouponl2dx} by employing a standard coercive form $(\mathcal{B}, H^{1,2}_0(U))$ on \( L^2(U) \) and constructing the associated semigroup with respect to the Lebesgue measure. Next, we characterize weak solutions to equation \eqref{bvpara} via an equivalent integral formulation (see Theorem~\ref{equiproweakso}). The core idea of the paper is then to transform the equation by introducing a weighted measure \( \mu = \rho\, dx \) so that the drift coefficient becomes divergence-free with respect to \( \mu \) as in \eqref{divfreecond}. This strategy, originally introduced in \cite{L25jm}, is further developed in Theorems~\ref{existinvar} and~\ref{leetransform}, and elaborated in Remark~\ref{leetransformrem}. Based on this transformation, we construct the associated Dirichlet form \( (\mathcal{E}, D(\mathcal{E})) \) on \( L^2(U, \mu) \) and derive a sub-Markovian \( C_0 \)-semigroup of contractions \( (T_t)_{t > 0} \) on \( L^2(U, \mu) \). This framework allows us to establish Theorem~\ref{weaksolsemi}, which forms the analytic foundation for our main result. Finally, to prove Theorem~\ref{maintheore}, we employ refined approximation techniques and compactness arguments, including weak convergence methods and the Banach–Saks theorem. The complete proof of Theorem~\ref{maintheore} is presented in detail after Remark~\ref{leetransformrem}. \\
As mentioned earlier, even when the zero-order term in \eqref{bvpara} vanishes, i.e., $c + \theta \equiv 0$, the positivity of the decay rate $\kappa= \frac{d^2 \lambda}{4K_1 (d - 1)^2 |U|^{2/d}}>0$ in Theorem~\ref{maintheore}(ii) remains valid. This guarantees that the exponential convergence persists robustly. In fact, the constant $K_1 \geq 1$, which appears in the expression of $\kappa$, originates from the elliptic Harnack inequality established in \cite{L25jm}. While its dependence on certain quantitative parameters is known, a more precise determination of the constant $K_1$ would potentially allow us to optimize the lower bound of $\kappa$.\\
While much of the recent literature on semigroup-based analysis in fractional and nonlocal settings (cf. \cite{Ar18, Sti19, CW20, DG21, C25}) focuses on integral-type operators, our work demonstrates that a semigroup-based approach can yield robust exponential decay even for local divergence-form equations with non-symmetric drift. This distinction highlights the broader applicability of semigroup theory and sets the stage for the analytic transformation introduced in this paper.
\\
In addition, let us consider the following non-divergence type initial-boundary value problem under the assumptions that $c + \theta \equiv 0$, $g \in C(\overline{U})$, and the drift coefficient $\hat{\mathbf{H}}$ is assumed to be smooth:
\begin{equation} \label{bvparanondiv}
\left\{
\begin{aligned}
    u_t -\frac12 \Delta u- \langle \hat{\mathbf{H}}, \nabla u \rangle  &= 0 && \text{in } U \times (0, T), \\
    u &= 0 && \text{on } \partial U \times [0, T], \\
    u &= g && \text{on } U \times \{t = 0\},
\end{aligned}
\right.
\end{equation}
As previously mentioned, it is  known that if $\hat{\mathbf{H}}$ is smooth, then the problem \eqref{bvparanondiv} admits a unique classical solution (see \cite[Chapter 2 and Proposition C.3.2]{Lo07}). Furthermore, assuming that $\hat{\mathbf{H}}$ satisfies a linear growth condition, we may consider the corresponding stochastic differential equation for each $x \in \mathbb{R}^d$ with a standard Brownian motion $(W_t)_{t \geq 0}$ on a probability space $(\Omega, \mathcal{F}, \mathbb{P})$:
\begin{equation*}
    X^x_t = x + W_t + \int_0^t \hat{\mathbf{H}}(X^x_s)\, ds, \quad t \geq 0, \quad \text{ $\mathbb{P}$-a.s.}
\end{equation*}
This SDE admits a unique strong solution $(X^x_t)_{t \geq 0}$ (\cite[Theorem 2.5.2]{Lo07}), and it follows from \cite[Theorem 2.4.4]{Lo07} that the associated probabilistic representation of the solution $u$ to \eqref{bvparanondiv} is given by
\begin{equation} \label{stochobjects}
    u(x, t) = \mathbb{E}\left[ g(X^x_t) \cdot \mathbf{1}_{\{t < \tau^x_U\}} \right], \quad (x, t) \in U \times (0, \infty),
\end{equation}
where $\tau^x_U := \inf \{ t \geq 0 : X^x_t \notin U \}$ is the exit time from the domain $U$ and $\mathbb{E}$ is the expectation with respect to $(\Omega, \mathcal{F}, \mathbb{P})$. Since \eqref{bvparanondiv} corresponds to the special case of \eqref{bvpara} with $A := \tfrac{1}{2} I$ and $\mathbf{H} := -\hat{\mathbf{H}}$, it follows from Theorem \ref{maintheore}(ii) that
\begin{equation*}
    \|u(\cdot, t)\|_{L^2(U)} \leq K_1^{\frac{1}{2}} e^{-\kappa t} \|g\|_{L^2(U)}.
\end{equation*}
This provides a quantitative estimate on the long-time behavior of the corresponding killed diffusion semigroup, demonstrating exponential decay in the $L^2$-norm. \\
This paper is organized as follows. Section~\ref{notconv} introduces the basic notions and conventions used throughout this paper. In Section~\ref{semigrouponl2dx}, we establish the existence and uniqueness of weak solutions to equation~\eqref{bvpara} by constructing a coercive bilinear form and the associated semigroup on the space \( L^2(U) \). This is done under a slightly more general assumption {\bf (Hy)} on the drift coefficient than {\bf (S)} required in the main result (Theorem \ref{maintheore}). Section~\ref{dfapproach} develops the analytic foundation necessary for our main results, where we construct a Dirichlet form and the associated semigroup with respect to a weighted measure. In Section~\ref{mainsection}, we apply the divergence-free transformation, which reformulates equation~\eqref{bvpara} into an equivalent equation with divergence-free drift coefficients. This transformation captures the core idea of the paper and leads to a detailed proof of the main result, Theorem~\ref{maintheore}. Finally, Section~\ref{discuss} presents further discussions and suggests directions for future research.

\section{Notations and conventions} \label{notconv}
Throughout this paper, we consider the Euclidean space $\mathbb{R}^d$, equipped with the standard inner product $\langle \cdot, \cdot \rangle$ and the corresponding Euclidean norm $\|\cdot\|$. For $x_0 \in \mathbb{R}^d$ and $r > 0$, we define $B_r(x_0) := \{ x \in \mathbb{R}^d : \|x - x_0\| < r \}$. We write $a \wedge b := \min\{a,b\}$ and $a \vee b := \max\{a,b\}$ for $a,b \in \mathbb{R}$, and denote by $1_W$ the indicator function of a set $W \subset \mathbb{R}^d$. The Lebesgue measure is denoted by $dx$, and $|E| := dx(E)$ for a Lebesgue measurable set $E$. Let $U \subset \mathbb{R}^d$ be an open set. We denote by $\mathcal{B}(U)$ the space of all Borel measurable functions $f : U \to \mathbb{R}$. For $\mathcal{A} \subset \mathcal{B}(U)$, we write $\mathcal{A}_0$ for the set of all functions in $\mathcal{A}$ with compact support in $U$.
$C(U)$ and $C(\overline{U})$ denote the sets of all continuous functions on $U$ and $\overline{U}$, respectively. 
We define $C_0(U) := C(U)_0$. For $k \in \mathbb{N} \cup \{\infty\}$, we write $C^k(U)$ for the space of functions with continuous derivatives up to order $k$, and $C_0^k(U) := C^k(U) \cap C_0(U)$. %; similarly, $C^k(\overline{U})$ denotes functions in $C^k(U)$ whose derivatives extend continuously to $\overline{U}$.  
Let $r \in [1, \infty]$, and let $L^r(U,\mu)$ denote the weighted $L^r$-space on $U$ with respect to a weighted measure $\mu$, equipped with the usual norm $\|\cdot\|_{L^r(U, \mu)}$. We write $L^r(U, dx)$ as $L^r(U)$. The space $L^r(U, \mathbb{R}^d, \mu)$ consists of $\mathbb{R}^d$-valued functions $\mathbf{F} = (f_1, \dots, f_d)$ with each $f_i \in L^r(U, \mu)$, and its norm is given by $\|\mathbf{F}\|_{L^r(U,\mu)} :=\big\| \| \mathbf{F} \| \big \|_{L^r(U, \mu)}$.
%We write $L^r_{\mathrm{loc}}(U)$ and $L^r_{\mathrm{loc}}(U, \mathbb{R}^d)$ for functions locally in $L^r$ over $U$.
Let $U$ be a bounded open subset of $\mathbb{R}^d$. The Sobolev space $H^{1,2}(U)$ consists of functions $f \in L^2(U)$ whose weak partial derivatives $\partial_i f$ belong to $L^2(U)$ for all $1 \leq i \leq d$, and it is equipped with the standard $H^{1,2}$-norm. %$\|f\|_{H^{1,2}(U)} := \|\nabla f\|_{L^2(U)}$, which is equivalent to the standard Sobolev norm by the Poincaré inequality.
We define $H_0^{1,2}(U)$ as the closure of $C_0^\infty(U)$ in $H^{1,2}(U)$, and define $H^{-1,2}(U)$ as the dual of $H_0^{1,2}(U)$. 
In particular, by the Poincaré inequality, we write $\| u\|_{H^{1,2}_0(U)} := \|  \nabla u \|_{L^2(U)}$ for all $u \in H^{1,2}_0(U)$.
%The second-order Sobolev space $H^{2,r}(U)$ consists of functions in $L^r(U)$ with all second-order weak derivatives $\partial_i \partial_j f \in L^r(U)$. 
The Laplacian is given by $\Delta f := \sum_{i=1}^d \partial_i^2 f$, and the Hessian by $\nabla^2 f := (\partial_i \partial_j f)_{1 \leq i,j \leq d}$.
Let $X$ be a Banach space. For $T > 0$ and $s \in [1, \infty]$, the Bochner space $L^s(0, T; X)$ consists of strongly measurable functions $\mathbf{u} : [0, T] \to X$ such that
\[
\|\mathbf{u}\|_{L^s(0,T;X)} := \left( \int_0^T \|\mathbf{u}(t)\|_X^s\,dt \right)^{1/s} < \infty,
\]
with the essential supremum norm used when $s = \infty$. The space $C([0,T]; X)$ denotes the set of continuous functions from $[0,T]$ into $X$, equipped with the supremum norm. We denote by $C([0,\infty); X)$ the set of all continuous functions $\mathbf{u} : [0, \infty) \to X$. For each $k \in \mathbb{N}$, the set $C^k((0,\infty); X)$ consists of all functions whose derivatives up to order $k$ exist and are continuous on $(0, \infty)$. We define
$C^\infty((0,\infty); X) := \bigcap_{k=1}^\infty C^k((0,\infty); X)$.

\section{Existence and uniqueness of weak solutions} \label{semigrouponl2dx}
In this section, we work under the following main hypothesis. \\ \\
\textbf{(Hy)}: $U$ is a bounded open subset of $\mathbb{R}^d$ with $d \geq 2$, $q \in (2, \infty)$ if $d=2$ and $q:=d$ if $d \geq 3$.
$\mathbf{H} \in L^q(U, \mathbb{R}^d)$ with $\|\mathbf{H}\| \leq h$ in $U$, where $h \in L^q(U)$, and $A = (a_{ij})_{1 \leq i,j \leq d}$ is a \textit{(possibly non-symmetric) matrix of measurable functions on $\mathbb{R}^d$} such that for some constants $M > 0$ and $\lambda > 0$,  \eqref{ellipticity} holds.
$\theta \in [0, \infty)$ is a constant, and $c \in L^{s}(U)$ satisfies $c \geq 0$ a.e. in $U$,  where $s \in (1, \infty)$ if $d=2$, and $s:=\frac{d}{2}$ if $d \geq 3$.\\
\centerline{}
%One can readily observe that the class of drift coefficients $\mathbf{H}$ in {\bf (Hy)} is strictly broader than that assumed in the structural condition {\bf (S)} introduced earlier.
The class of drift coefficients $\mathbf H$ in {\bf (Hy)} is broader than that in {\bf (S)}, and is strictly broader when $d\geq3$.
\centerline{}
%%%%%%%%%%%%%%
The following proposition may originate from Stampacchia, \cite{St65} (cf \cite{T73}), and {\bf (Hy)} is widely used to guarantee coercivity of the bilinear form associated with a linear elliptic operator. In particular, we explicitly compute and present the constants appearing in the coercivity inequalities.
\begin{prop} \label{stamenest}
Assume that {\bf (Hy)} holds. Define a bilinear form $\mathcal{B}: H^{1,2}_0(U) \times H^{1,2}_0(U) \rightarrow \mathbb{R}$ given by
$$
\mathcal{B}(f,g)= \int_{U} \langle A \nabla f, \nabla g \rangle  +\langle \mathbf{H}, \nabla f \rangle g+(c+\theta) f g \,dx, \quad f,g \in H^{1,2}_0(U).
$$
Then, the following hold:
\begin{itemize}
\item[(i)] 
Let $K>0$ be a constant given by
\[
K := 
\begin{cases}
dM + \dfrac{q}{q-2} |U|^{\frac{1}{2} - \frac{1}{q}} \|h\|_{L^q(U)} 
+ \left( {\dfrac{s}{s - 1} }\right)^2 |U|^{1 - \frac{1}{s}} \|c+\theta\|_{L^{s}(U)}, 
& \text{if } d = 2, \\[12pt]
dM + \dfrac{2(d-1)}{d-2} \|h\|_{L^d(U)} 
+ \left( \dfrac{2(d-1)}{d-2} \right)^2 \|c+\theta\|_{L^{\frac{d}{2}}(U)}, 
& \text{if } d \geq 3.
\end{cases}
\]
Then, 
$$
|\mathcal{B}(f,g)| \leq K \| \nabla f \|_{L^2(U)}  \| \nabla g \|_{L^2(U)} \quad \text{ for all $f,g \in H^{1,2}_0(U)$}.
$$
\item[(ii)]
Let $N \geq 0$ be a constant such that
\[
\begin{array}{ll}
\displaystyle
\left( \int_U 1_{\{ |h| \geq N \}} |h|^q \, dx \right)^{\frac{2}{q}} 
\leq \frac{\lambda^2}{4} \left( \frac{q-2}{q} \right)^2 |U|^{-1 + \frac{2}{q}}, 
& \text{if } d = 2, \\[12pt]
\displaystyle
\left( \int_U 1_{\{ |h| \geq N \}} |h|^d \, dx \right)^{\frac{2}{d}} 
\leq \frac{\lambda^2}{16} \left( \frac{d-2}{d-1} \right)^2, 
& \text{if } d \geq 3.
\end{array}
\]
Then,
\[
\mathcal{B}(f,f) +\frac{N^2}{\lambda} \| f\|^2_{L^2(U)}  \geq \frac{\lambda}{2}\|\nabla f \|^2_{L^2(U)} \quad \text{ for all $f \in H^{1,2}_0(U)$}.
\]
\end{itemize}
\end{prop}
\begin{proof}
(i) \underline{\sf Case 1}) Assume that $d=2$. Let $f,g \in H^{1,2}_0(U)$.
First,  by Cauchy-Schwarz inequality,
\[
\left| \int_{U} \langle A \nabla f, \nabla g \rangle dx \right| \leq \int_{U} dM \| \nabla f \|  \| \nabla g \| dx  \leq d M \| \nabla f \|_{L^2(U)} \| \nabla g \|_{L^2(U)}.
\]
Next, by applying Hölder's inequality, we obtain
\[
\left | \int_{U} \langle \mathbf{H}, \nabla f \rangle g dx \right| \leq \int_{U} |h|\,  \| \nabla f \|\, |g| dx \leq \|h\|_{L^q(U)} \| \nabla f \|_{L^2(U)} \| g \|_{L^{\frac{2q}{q-2}}(U)}.
\]
By Sobolev's inequality in \cite[Section 5.6, Theorem 1]{E10},
\begin{equation} \label{sobolconsta}
\| g \|_{L^{\frac{2q}{q-2}}(U)} \leq \left(  \frac{\frac{q}{q-1}}{2-\frac{q}{q-1}}  \right)
 \|\nabla g \|_{L^\frac{q}{q-1}(U)} \leq \frac{q}{q-2}  |U|^{\frac{1}{2}-\frac{1}{q}}
 \|\nabla g \|_{L^2(U)}.
\end{equation}
Similarly, using H\"{o}lder's inequality,
\[
\left |\int_{U}(c+\theta)  f g\, dx \right| \leq  \int_{U} |c+\theta|  |f| |g|\, dx \leq \|c+\theta \|_{L^{s}(U)} \|f\|_{L^{\frac{2s}{s-1}}(U)} \|g\|_{L^{\frac{2s}{s-1}}(U)}.
\]
Again by Sobolev's inequality in \cite[Section 5.6, Theorem 1]{E10} and H\"{o}lder's inequality,
$$ 
\|f\|_{L^{\frac{2s}{s-1}}(U)} \leq \left( \frac{\frac{2s}{2s-1}}{2-\frac{2s}{2s-1}}  \right)  \| \nabla f\|_{L^{\frac{2s}{2s-1}}(U)} \leq \left( \frac{s}{s-1} \right) |U|^{\frac12-\frac{1}{2s}} \| \nabla f\|_{L^2(U)}.
$$
Therefore, from the above, the estimate in (i) is verified when $d=2$. \\
\underline{\sf Case 2}) Assume that $d \geq 3$. Let $f,g \in H^{1,2}_0(U)$. 
Applying Sobolev's inequality in \cite[Section 5.6, Theorem 1]{E10} (cf. \cite[Theorem 4.8]{EG15}),
\begin{align} \label{imptsobolmud}
\|f\|_{L^{\frac{2d}{d-2}}(U)} \leq  \frac{2(d-1)}{d-2}	 \| \nabla f \|_{L^2(U)}.
\end{align}
Then, using Cauchy–Schwarz inequality and H\"{o}lder's inequality and \eqref{imptsobolmud},
\begin{align*}
\left | \int_{U} \langle \mathbf{H}, \nabla f \rangle g dx \right| &\leq \|h\|_{L^d(U)} \| \nabla f \|_{L^2(U)} \| g \|_{L^{\frac{2d}{d-2}}(U)}  \leq \frac{2(d-1)}{d-2}	  \|h\|_{L^d(U)}\| \nabla f \|_{L^2(U)}\| \nabla g \|_{L^2(U)},
\end{align*}
and
\begin{align*}
\left |\int_{U}(c+\theta)  f g\, dx \right| &\leq \|c+\theta \|_{L^{\frac{d}{2}}(U)} \|f\|_{L^{\frac{2d}{d-2}}(U)} \|g\|_{L^{\frac{2d}{d-2}}(U)} \\
& \leq  \|c+\theta \|_{L^{\frac{d}{2}}(U)} \left(  \frac{2(d-1)}{d-2}	\right)^2 \| \nabla f \|_{L^2(U)}\| \nabla g \|_{L^2(U)}.
\end{align*}
This completes the proof for the case $d \geq 3$. \\
\centerline{}
(ii) \underline{\sf Case 1}) Assume that $d=2$. Let $f \in H^{1,2}_0(U)$. Then, using the Cauchy–Schwarz inequality, Young's inequality, H\"{o}lder's inequality and Sobolev's inequality,
\begin{align} \label{driftestimd2}
\left | \int_{U} \langle \mathbf{H}, \nabla f \rangle f\,dx \right| &\leq \int_{U} \|\mathbf{H} \| \,|f|\, \| \nabla f \|\,dx \leq \frac{\lambda}{4} \int_{U} \| \nabla f \|^2 dx + \frac{1}{\lambda} \int_{U} h^2 |f|^2 dx.
\end{align}
Define a function $\phi_1: [0, \infty) \rightarrow \mathbb{R}$ given by
$$
\phi_1(s):= \left( \int_{U} 1_{\{  |h| \geq s \}} |h|^q dx \right)^{\frac{2}{q}}, \quad s \in [0, \infty).
$$
Since $h \in L^q(U)$ and $h$ is almost everywhere finite, it follows from Lebesgue's theorem that $\lim_{s \rightarrow \infty}\phi_1(s)=0$. Now choose a constant $N \geq 0$ satisfying 
\begin{equation} \label{imptindex}
\phi_1(N) \leq \frac{\lambda^2}{4} \left(\frac{q-2}{q} \right)^2  |U|^{-1+\frac{2}{q}}.
\end{equation}
Then, using \eqref{sobolconsta} and \eqref{imptindex},
\begin{align*}
 \int_{U} h^2 |f|^2 dx &= \int_{U} 1_{\{ |h| \geq N\}}h^2 |f|^2 dx +  \int_{U} 1_{\{ |h| < N\}}h^2 |f|^2 dx \\
&\leq \left(\int_{U} 1_{\{  |h| \geq N \}} |h|^q dx \right)^{\frac{2}{q}} \|f\|_{L^{\frac{2q}{q-2}}(U)}^2 +N^2 \int_{U} |f|^2 dx \\
& \leq \frac{\lambda^2}{4} \|\nabla f \|^2_{L^2(U)}+N^2 \int_{U} |f|^2 dx.
\end{align*}
Therefore, \eqref{driftestimd2} yields 
\begin{align*}
 \int_{U} \langle \mathbf{H}, \nabla f \rangle f\,dx  \geq -\frac{\lambda}{2} \int_{U} \| \nabla f \|^2 dx -\frac{N^2}{\lambda} \int_{U} |f|^2 dx.
\end{align*}
Therefore, combining the above estimates, we obtain 
\[
\mathcal{B}(f,f)  \geq \frac{\lambda}{2} \|\nabla f \|^2_{L^2(U)} -\frac{N^2}{\lambda} \|f\|_{L^2(U)}^2,
\]
as required. \\
 \underline{\sf Case 2}) 
Assume that $d \geq 3$. Let $f \in H^{1,2}_0(U)$. Analogously to \textsf{Case 1}, we obtain \eqref{driftestimd2}.
Define a function $\phi_2: [0, \infty) \rightarrow \mathbb{R}$ given by
$$
\phi_2(s):= \left( \int_{U} 1_{\{  |h| \geq s \}} |h|^d dx \right)^{\frac{2}{d}}, \quad s \in [0, \infty).
$$
Since $h \in L^d(U)$ and $h$ is almost everywhere finite, it follows from Lebesgue's theorem that $\lim_{s \rightarrow \infty}\phi_2(s)=0$. Now choose $N \geq 0$ satisfying 
\begin{equation} \label{imptindex2}
\phi_2(N) \leq \frac{\lambda^2}{16} \left(\frac{d-2}{d-1} \right)^2 .
\end{equation}
From \eqref{imptsobolmud} and \eqref{imptindex2}, we get

\begin{align*}
 \int_{U} h^2 |f|^2 dx &= \int_{U} 1_{\{ |h| \geq N\}}h^2 |f|^2 dx +  \int_{U} 1_{\{ |h| < N\}}h^2 |f|^2 dx \\
&\leq \left(\int_{U} 1_{\{  |h| \geq N \}} |h|^d dx \right)^{\frac{2}{d}} \|f\|_{L^{\frac{2d}{d-2}}(U)}^2 +N^2 \int_{U} |f|^2 dx \\
& \leq \frac{\lambda^2}{4} \|\nabla f \|^2_{L^2(U)}+N^2 \int_{U} |f|^2 dx.
\end{align*}
The rest is analogous to {\sf Case 1}.
\end{proof}
\centerline{}
\noindent
We now introduce the basic definition of a weak solution to the initial-boundary value problem \eqref{bvpara}, formulated as a variational identity involving test functions. Moreover, since the initial data is given in $L^2(U)$, the initial condition is interpreted in the sense of an $L^2$-limit.
\begin{defn} \label{startdefn}
Assume that {\bf (Hy)} holds. Let $T \in (0, \infty)$ and let $u$ be a function defined on $U \times (0,T)$ such that $u(\cdot,t) \in H^{1,2}_0(U)$ for a.e. $t \in (0,T)$, and define
%$$
%\mathbf{u} : (0,T) \to H_0^{1,2}(U)
%$$
%by
\begin{equation} \label{cylderfun}
\mathbf{u}(t) := u(\cdot,t) \quad \text{for a.e. } t \in (0,T).
\end{equation}
Let $g \in L^2(U)$. We say that $\mathbf{u}$
is a weak solution to \eqref{bvpara},
if the following hold (cf. Theorem \ref{equiproweakso}):
\begin{itemize}
    \item[(i)]
    \[
	\mathbf{u} \in L^2(0,T; H_0^{1,2}(U)) \quad \text{and} \quad \mathbf{u}' \in L^2(0,T; H^{-1,2}(U))
     \]
    \item[(ii)] 
    \begin{equation} \label{weaksolmain}
    \langle \mathbf{u}'(t), v \rangle_{H^{-1,2}(U)} + \int_{U} \langle A \nabla \mathbf{u}(t), \nabla v \rangle  +\langle \mathbf{H}, \nabla \mathbf{u}(t) \rangle v+(c+\theta)\mathbf{u}(t)v \,dx = 0 \quad \text{for all } v \in H_0^{1,2}(U) \text{ and a.e. } t \in (0,T)
    \end{equation}
    \item[(iii)] 
    \[
	\lim_{t \rightarrow 0+} \mathbf{u}(t)=g \quad \text{ in $L^2(U)$}.
    \]
\end{itemize}
\end{defn}

\begin{rem} \label{imptremak3}
Note that a weak solution $\mathbf{u}$ in Definition \ref{startdefn} is well-defined by Proposition \ref{stamenest}. As a direct consequence of regularity results for Sobolev space-valued functions (see \cite[Section 5.9, Theorem 3(i)]{E10}), if $\mathbf{u} \in L^2(0,T; H_0^{1,2}(U))$ with $\mathbf{u}' \in L^2(0,T; H^{-1,2}(U))$, then $\mathbf{u}$ has a version in $C([0,T]; L^2(U))$, say again $\mathbf{u} \in C([0,T]; L^2(U)) \cap L^2(0,T; H^{1,2}_0(U))$. Therefore, if $u$ is a weak solution to \eqref{bvpara}, then the condition in Definition \ref{startdefn}(iii) is replaced as
$$
\mathbf{u}(0)=\lim_{t \rightarrow 0+} \mathbf{u}(t)=g \quad \text{ in $L^2(U)$}.
$$
Moreover, by \cite[Section 5.9, Theorem 3(ii)]{E10},
the map
$$
t \mapsto \| \mathbf{u}(t)\|^2_{L^2(U)}\;\; \text{ is absolutely continuous on $[0, T]$}
$$
and satisfies
\begin{equation} \label{timederih12}
\frac{d}{dt} \| \mathbf{u}(t)\|^2_{L^2(U)} =2 \langle \mathbf{u}'(t),  \mathbf{u}(t)  \rangle_{H^{-1,2}(U)} \quad \text{ for a.e. $t \in (0,T)$}.
\end{equation}
\end{rem}
\centerline{}
\noindent
In the following result, we present an equivalent formulation of weak solutions, in which the condition $\mathbf{u}' \in L^2(0,T; H^{-1,2}(U))$ is not explicitly assumed in (iii).
\begin{theo} \label{equiproweakso}
Assume that {\bf (Hy)} holds. Let $u$ be a function defined on $U \times (0,T)$ such that $u(\cdot, t) \in H^{1,2}_0(U)$ for a.e. $t \in (0,T)$.
Define $\mathbf{u}$ as in \eqref{cylderfun} and assume that $\mathbf{u} \in L^2(0,T; H^{1,2}_0(U))$. Then, the following are equivalent:
\begin{itemize}
\item[(i)]
$\mathbf{u}' \in L^2(0,T;H^{-1,2}(U))$ and \eqref{weaksolmain} holds
\item[(ii)] $\mathbf{u}' \in L^2(0,T;H^{-1,2}(U))$ and for all $\mathbf{v} \in L^2(0,T; H_0^{1,2}(U))$,
    \begin{equation} \label{intformv}
    \int_0^T \langle \mathbf{u}'(t), \mathbf{v}(t) \rangle_{H^{-1,2}(U)} \,dt 
    + \int_0^T \int_U \langle A \nabla \mathbf{u}(t), \nabla \mathbf{v}(t) \rangle 
    + \langle \mathbf{H}, \nabla \mathbf{u}(t) \rangle \mathbf{v}(t) 
    + (c+\theta) \mathbf{u}(t) \mathbf{v}(t) \,dx \,dt = 0.
    \end{equation}
\item[(iii)] For all $v \in H^{1,2}_0(U)$ and $\eta \in C_0^{\infty}((0,T))$,
\begin{equation} \label{iterativeint}
\int_0^T \int_{U} -\mathbf{u}(t) \partial_t (v \eta) + \langle A \nabla \mathbf{u}(t), \nabla (v \eta) \rangle 
+ \langle \mathbf{H}, \nabla \mathbf{u}(t) \rangle (v \eta) + (c+\theta) \mathbf{u}(t) (v \eta) \,dx \,dt = 0.
\end{equation}
\end{itemize}
\end{theo}
\begin{proof}
(i) $\Rightarrow$ (ii): Let $\tilde{\mathbf{v}} \in C([0,T]; H_0^{1,2}(U))$.
Then, by integrating $\tilde{\mathbf{v}}$ over $[0,T]$ with respect to $dt$ in (i), we obtain \eqref{intformv} with $\mathbf{v}$ replaced by $\tilde{\mathbf{v}}$. By applying Young's inequality, Proposition \ref{stamenest}, and the density of $C([0,T]; H^{1,2}_0(U))$ in $L^2(0,T; H^{1,2}_0(U))$, the assertion follows.
\\ 
\centerline{}
(ii) $\Rightarrow$ (iii): Let $v \in H^{1,2}_0(U)$ and $\eta \in C_0^{\infty}((0,T))$.
Define $\mathbf{v} : [0,T] \to H^{1,2}_0(U)$ by $\mathbf{v}(t)(x) := v(x)\eta(t)$ for $x \in U$ and $t \in [0,T]$. Then,
\begin{align*}
\int_0^T \langle \mathbf{u}'(t), \mathbf{v}(t) \rangle_{H^{-1,2}(U)} \,dt 
&= \left\langle \int_0^T \mathbf{u}'(t) \eta(t) \,dt, v \right\rangle_{H^{-1,2}(U)} \\
&= \left\langle -\int_0^T \mathbf{u}(t) \eta'(t) \,dt, v \right\rangle_{H^{-1,2}(U)} \\
&= -\int_0^T \langle \mathbf{u}(t), v \eta'(t) \rangle_{H^{-1,2}(U)} \,dt \\
&=\int_0^T \int_{U} -\mathbf{u}(t) \partial_t (v \eta) dx dt.
\end{align*}
Hence, (ii) implies \eqref{iterativeint}. \\
\centerline{}
(iii) $\Rightarrow$ (i):   For a.e. $t\in(0,T)$, define $\tilde{\mathbf{w}}(t) \in H^{-1,2}(U)$ given by
\begin{equation} \label{weaktimedri}
\langle \tilde{\mathbf{w}}(t), \varphi \rangle_{H^{-1,2}(U)} =-\int_{U}\langle A \nabla \mathbf{u}(t), \nabla \varphi \rangle 
+ \langle \mathbf{H}, \nabla \mathbf{u}(t) \rangle \varphi + (c+\theta) \mathbf{u}(t)  \varphi \,dx, \quad \varphi \in H^{1,2}_0(U).
\end{equation}
Then, by Proposition \ref{stamenest}(i), for a.e. $t \in (0, T)$
\begin{align*}
\left| \langle \tilde{\mathbf{w}}(t), \varphi \rangle_{H^{-1,2}(U)} \right| &\leq K \| \nabla \mathbf{u}(t) \|_{L^2(U)} \| \nabla \varphi \|_{L^2(U)} \quad \text{ for all } \varphi \in H^{1,2}_0(U),
\end{align*}
where $K > 0$ is the constant appearing in Proposition \ref{stamenest}(i).
Thus, 
$$
\|\tilde{\mathbf{w}}(t)\|_{H^{-1,2}(U)} \leq K \| \nabla \mathbf{u}(t) \|_{L^2(U)} \quad \text{ for a.e. $t \in (0,T)$},
$$
so that $\tilde{\mathbf{w}} \in L^2(0,T; H^{-1,2}(U))$, and satisfies
\begin{equation} \label{timederiest}
\int_0^T \| \tilde{\mathbf{w}}(t)\|^2_{H^{-1,2}(U)} dt \leq K^2 \|\mathbf{u}\|^2_{L^2(0,T; H^{1,2}_0(U))}. 
\end{equation}
 Let $v \in H^{1,2}_0(U)$ and $\eta \in C_0^{\infty}((0,T))$. Using the Bochner integration result in \cite[Appendices, Theorem 8]{E10}, we obtain:
\begin{align}
&\left \langle -\int_{0}^T \mathbf{u}(t)  \eta'(t) dt,   v \right \rangle_{H^{-1,2}(U)}=-\int_0^T \langle \mathbf{u}(t), v \eta'(t) \rangle_{H^{-1,2}(U)} \,dt=-\int_0^T \int_U \mathbf{u}(t) \, \partial_t (v \eta) \,dx \,dt  \nonumber \\
&= -\int_{0}^T \int_{U} \langle A \nabla \mathbf{u}(t), \nabla (v \eta) \rangle 
+ \langle \mathbf{H}, \nabla \mathbf{u}(t) \rangle (v \eta) + (c+\theta) \mathbf{u}(t) (v \eta) \,dx \,dt \nonumber \\
&= \int_0^T \left(- \int_{U}\langle A \nabla \mathbf{u}(t), \nabla v\rangle 
+ \langle \mathbf{H}, \nabla \mathbf{u}(t) \rangle v + (c+\theta) \mathbf{u}(t)  v\,dx \right)\, \eta(t) dt \nonumber \\
&= \int_0^T  \langle \tilde{\mathbf{w}}(t), v \rangle_{H^{-1,2}(U)} \eta(t)dt	 =  \left \langle  \int_0^T \tilde{\mathbf{w}}(t) \eta(t)\,dt, v  \right  \rangle_{H^{-1,2}(U)}. \label{soboldual}
\end{align}
Therefore, since $v$ and $\eta$ are arbitrarily chosen, $\mathbf{u}'=\tilde{\mathbf{w}} \in L^2(0,T; H^{-1,2}(U))$. Moreover, from \eqref{soboldual} we obtain that
\begin{align*}
-\int_{0}^T \int_U \mathbf{u}(t) \, \partial_t (v \eta) \,dx \,dt = \int_0^T \langle \mathbf{u}'(t), v \rangle_{H^{-1,2}(U)} \eta(t) \,dt.
\end{align*}
Similarly from \eqref{soboldual}, we have
\begin{align*}
&\int_0^T \int_{U} \langle A \nabla \mathbf{u}(t), \nabla (v \eta) \rangle 
+ \langle \mathbf{H}, \nabla \mathbf{u}(t) \rangle (v \eta) + (c+\theta) \mathbf{u}(t) (v \eta) \,dx \,dt \\
& \quad = \int_0^T \left( \int_U \langle A \nabla \mathbf{u}, \nabla v \rangle 
+ \langle \mathbf{H}, \nabla \mathbf{u} \rangle v + (c+\theta) \mathbf{u}(t) v \,dx \right) \eta(t) \,dt.
\end{align*}
Thus,
\[
\int_0^T \left( \langle \mathbf{u}'(t), v \rangle_{H^{-1,2}(U)} 
+ \int_U \langle A \nabla \mathbf{u}(t), \nabla v \rangle 
+ \langle \mathbf{H}, \nabla \mathbf{u}(t) \rangle v 
+ (c+\theta) \mathbf{u}(t) v \,dx \right) \eta(t) \,dt = 0.
\]
Since $\eta \in C_0^{\infty}((0,T))$ is arbitrary, we conclude that (i) holds.
\end{proof}

\begin{cor} \label{cor:timemaxcontrol}
Under the assumptions of {\bf (Hy)}, assume that $\mathbf{u}$ is a weak solution to \eqref{bvpara}. Then the following  estimates hold:
\begin{equation} \label{interestim}
\|\mathbf{u}' \|^2_{L^2(0,T; H^{-1,2}(U))} \leq K^2\|\mathbf{u}\|^2_{L^2(0,T; H^{1,2}_0(U))},
\end{equation}
\begin{equation} \label{eq:timemax}
\max_{t \in [0,T]}\|\mathbf{u}(t)\|_{L^2(U)}^2  \leq   \left( \frac{4(d-1)^2}{d^2}|U|^{\frac2d}T^{-1}+K^2+1 \right)\|\mathbf{u}\|^2_{L^2(0,T; H^{1,2}_0(U))}.
\end{equation}
where $K > 0$ is the constant appearing in Proposition \ref{stamenest}(i).
\end{cor}
\begin{proof}
From Proposition \ref{stamenest}(i) and \eqref{timederiest}, we have \eqref{interestim}.
Observe that applying Sobolev's inequality in \cite[Section 5.6, Theorem 1]{E10} and the H\"{o}lder inequality, we obtain that
\begin{align}
\|f\|_{L^2(U)} \leq \frac{\frac{2d}{d+2} (d-1)}{d-\frac{2d}{d+2}} \| \nabla f \|_{L^{\frac{2d}{d+2}}(U)} \leq \frac{2(d-1)}{d} |U|^{\frac{1}{d}} \| \nabla f\|_{L^2(U)} \quad \text{ for all $f \in H^{1,2}_0(U)$.} \label{sobolhold}
\end{align}
By \cite[Section 5.9 (12)]{E10}, the following energy identity holds for all $s,t \in [0,T]$:
\[
\|\mathbf{u}(t)\|_{L^2(U)}^2 = \|\mathbf{u}(s)\|_{L^2(U)}^2 + 2 \int_s^t \langle \mathbf{u}'(\tau), \mathbf{u}(\tau) \rangle_{H^{-1,2}(U)} \, d\tau.
\]
Applying Young's inequality and \eqref{interestim}, we have for any $s,t \in [0,T]$,
\begin{align*}
\|\mathbf{u}(t)\|_{L^2(U)}^2 
&\leq \|\mathbf{u}(s)\|_{L^2(U)}^2 
+ \int_0^T \|\mathbf{u}'(\tau)\|^2_{H^{-1,2}(U)} \, d\tau
+ \int_0^T \|\mathbf{u}(\tau)\|^2_{H_0^{1,2}(U)} \, d\tau \\
&= \|\mathbf{u}(s)\|_{L^2(U)}^2 
+ (K^2 + 1) \|\mathbf{u}\|^2_{L^2(0,T; H^{1,2}_0(U))}.
\end{align*}
Integrating this inequality with respect to $s$ over $[0,T]$ and using \eqref{sobolhold}, we get
\begin{align*}
T \|\mathbf{u}(t)\|_{L^2(U)}^2 
&\leq \int_0^T \|\mathbf{u}(s)\|_{L^2(U)}^2 \, ds 
+ (K^2 + 1) T \|\mathbf{u}\|^2_{L^2(0,T; H^{1,2}_0(U))} \\
&\leq \left(\frac{4(d-1)^2}{d^2}|U|^{\frac2d}+	TK^2+T\right) \|\mathbf{u}\|^2_{L^2(0,T; H^{1,2}_0(U))},
\end{align*}
and hence \eqref{eq:timemax} follows.
\end{proof}

\begin{theo}[Uniqueness] \label{theouniquene}
Assume that {\bf (Hy)} holds. Let $\mathbf{u}$ and $\mathbf{v}$ be weak solutions to \eqref{bvpara}. Then $\mathbf{u} = \mathbf{v}$ in $L^2(0,T; H^{1,2}_0(U))$.
\end{theo}
\begin{proof}
As in Remark \ref{imptremak3}, let $\mathbf{w} := \mathbf{u} - \mathbf{v} \in C([0,T]; L^2(U)) \cap L^2(0,T; H^{1,2}_0(U))$. Then $\mathbf{w}$ is also a weak solution to \eqref{bvpara} due to the linearity of the equation. For a.e.\ $t \in [0,T]$, the weak formulation yields
\begin{equation} \label{intformv2}
\langle \mathbf{w}'(t), \mathbf{w}(t) \rangle_{H^{-1,2}(U), H^{1,2}_0(U)} 
+ \int_U \langle A \nabla \mathbf{w}(t), \nabla \mathbf{w}(t) \rangle 
+ \langle \mathbf{H}, \nabla \mathbf{w}(t) \rangle \mathbf{w}(t) 
+ (c+\theta)\, \mathbf{w}(t)^2  dx = 0.
\end{equation}
Applying \eqref{timederih12} and Proposition \ref{stamenest}(ii) to \eqref{intformv2}, we obtain the estimate
\begin{equation} \label{eq:energyineq}
\frac{1}{2} \frac{d}{dt} \left( \| \mathbf{w}(t) \|_{L^2(U)}^2 \right)
-\frac{N^2}{\lambda} \| \mathbf{w}(t) \|_{L^2(U)}^2 \leq 0
\quad \text{for a.e.\ } t \in [0,T],
\end{equation}
where $N \geq 0$ is the constant appearing in Proposition \ref{stamenest}(ii).
Define $\Phi \in C([0,T])$ given by
\[
\Phi(t) := \| \mathbf{w}(t) \|_{L^2(U)}^2, \quad t \in [0,T],
\]
and let $\gamma := \frac{N^2}{\lambda}$. Then \eqref{eq:energyineq} implies
\[
\frac{d}{dt} \left( e^{-2\gamma t} \Phi(t) \right) \leq 0
\quad \text{for a.e.\ } t \in [0,T].
\]
Integrating from $0$ to $t$, we get
\begin{equation} \label{eq:phidecay}
e^{-2\gamma t} \Phi(t) \leq \Phi(0).
\end{equation}
Since $\Phi(0) = \| \mathbf{w}(0) \|_{L^2(U)}^2 = 0$, it follows from \eqref{eq:phidecay} that $\Phi(t) = 0$ for all $t \in [0,T]$. Hence, $\mathbf{w} = 0$ in $C([0,T]; L^2(U))$. Therefore, $\mathbf{u} = \mathbf{v}$ in $L^2(0,T; H^{1,2}_0(U))$, as desired.
\end{proof}
\centerline{}
\noindent
The following not only provides existence results, but also identifies the solution via an analytic semigroup and derives corresponding $L^2$ and energy estimates.
\begin{theo} \label{mainwellposth}
Assume that {\bf (Hy)} holds and let $T \in (0, \infty)$. Let $N \geq 0$ be a constant appearing in Proposition \ref{stamenest}(ii), $\gamma:=\frac{N^2}{\lambda}$ and $g \in L^2(U)$. Then, the following hold:
%there exists a (unique) weak solution $\mathbf{u} \in L^2(0,T; H^{1,2}_0(U)) \cap C([0, \infty); L^2(U))$ to \eqref{bvpara} satisfying the following properties:
\begin{itemize}
\item[(i)]
There exists a $C_0$-semigroup of contractions $(S_t)_{t>0}$ on $L^2(U)$ such that the $L^2(U)$-space valued function $\mathbf{u} \in C([0, \infty); L^2(U))$ defined by
$$
\mathbf{u}(t)=e^{\gamma t} S_t g \quad \text{$t \in (0, \infty)$}, \qquad \mathbf{u}(0)=g
$$
is a (unique) weak solution to \eqref{bvpara}.
In particular, 
\begin{equation} \label{hcondil2est}
\|\mathbf{u}(t)\|_{L^2(U)} \leq e^{\gamma t} \|g\|_{L^2(U)} \quad \text{ for all $t \in [0,\infty)$}
\end{equation}
and
\begin{equation} \label{hcondienergy}
\| \mathbf{u}\|_{L^2(0,T; H^{1,2}_0(U))} 
\leq \left( \frac{e^{2\gamma T} }{\lambda} \right)^{1/2} \|g\|_{L^2(U)}.
\end{equation}
\item[(ii)]
$\mathbf{u}$ in (i) satisfies $\mathbf{u} \in C^{\infty}((0,\infty); L^2(U))$.
\end{itemize}
\end{theo}
\begin{proof}
(i)
Define a bilinear form $\mathcal{B}_{\gamma}: H^{1,2}_0(U) \times H^{1,2}_0(U) \rightarrow \mathbb{R}$ given by
\[
\mathcal{B}_{\gamma}(f,g)= \int_{U} \langle A \nabla f, \nabla g \rangle  +\langle \mathbf{H}, \nabla f \rangle g+(c+\theta+\gamma) f g \,dx, \quad f,g \in H^{1,2}_0(U).
\]
By Proposition \ref{stamenest}(ii) and \eqref{sobolhold}, we obtain
\begin{align} 
\frac{\lambda}{2} \| \nabla f \|^2_{L^2(U)} \leq \mathcal{B}_{\gamma}(f,f) &\leq K \| \nabla f \|_{L^2(U)}^2 + \gamma \| f \|_{L^2(U)}^2 \nonumber  \\
&\leq \left(K +  \gamma \frac{4(d-1)^2}{d^2} |U|^{\frac{2}{d}} \right)  \| \nabla f \|^2_{L^2(U)} \quad \text{ for all $f \in C_0^{\infty}(U)$}, \label{comparestib}
\end{align}
where $K>0$ is a constant as in Proposition \ref{stamenest}(i). 
Since $H^{1,2}_0(U)$ is complete, it follows from \eqref{comparestib} that it is also complete with respect to the norm $\mathcal{B}_{\gamma}(\cdot, \cdot)^{1/2} + \| \cdot \|_{L^2(U)}$.
In addition, by Proposition \ref{stamenest}(i) and \eqref{sobolhold},
\begin{align*}
|\mathcal{B}_{\gamma}(f,g)| &\leq K \| \nabla f \|_{L^2(U)} \| \nabla g \|_{L^2(U)} + \gamma \| f \|_{L^2(U)} \| g\|_{L^2(U)} \\
&\leq K \| \nabla f \|_{L^2(U)} \| \nabla g \|_{L^2(U)} + \gamma \frac{4(d-1)^2}{d^2} |U|^{\frac{2}{d}} \| \nabla f \|_{L^2(U)} \| \nabla g \|_{L^2(U)} \\
&\leq \frac{2}{\lambda} \left(K + \gamma \frac{4(d-1)^2}{d^2} |U|^{\frac{2}{d}} \right) \mathcal{B}_{\gamma}(f,f)^{1/2} \mathcal{B}_{\gamma}(g,g)^{1/2} \quad \text{for all $f,g \in H^{1,2}_0(U)$}.
\end{align*}
Accordingly, $\mathcal{B}_{\gamma}$ satisfies the strong sector condition (see \cite[Chapter I, Section 2]{MR92}). 
Hence, $(\mathcal{B}_{\gamma}, H^{1,2}_0(U))$ is a coercive closed form on $L^2(U)$. 
Let $(R_{\alpha})_{\alpha>0}$ be a strongly continuous contraction resolvent on $L^2(U)$ associated with $(\mathcal{B}_{\gamma}, H^{1,2}_0(U))$ as in \cite[Chapter I, Theorem 2.8]{MR92}, and let $(S_t)_{t>0}$ be its corresponding $C_0$-semigroup of contractions on $L^2(U)$ according to Hille--Yosida's theorem (\cite[Chapter I, Theorem 1.12]{MR92}). 
Let $(\mathcal{L}, D(\mathcal{L}))$ be the generator associated to $(S_t)_{t>0}$ as in \cite[Chapter I, Definition 1.8]{MR92}. 
First, let $f \in D(\mathcal{L})$. Then $S_t f \in D(\mathcal{L}) \subset H^{1,2}_0(U)$ for all $t>0$. Moreover, by \cite[Chapter I, Corollary 2.10]{MR92}, we obtain
\begin{equation} \label{semigenolid}
\mathcal{B}_{\gamma}(S_t f, \varphi) = -\int_{U} \mathcal{L} S_t f  \cdot \varphi \,dx = -\int_{U} \partial_t S_t f \cdot \varphi \,dx \quad \text{for all $\varphi \in H^{1,2}_0(U)$}
\end{equation}
and
\begin{align*}
\frac{\lambda}{2} \| \nabla S_t f \|^2_{L^2(U)} \leq \mathcal{B}_{\gamma}(S_tf, S_t f) = -\int_{U} \mathcal{L} S_tf \cdot S_t f \,dx \leq \| \mathcal{L}f \|_{L^2(U)} \|f \|_{L^2(U)} \quad \text{for all $t \in (0, \infty)$}.
\end{align*}
For each $\phi \in L^2(U)$, define $\mathbf{u}^{\phi}:[0,\infty) \rightarrow L^2(U)$ by
\[
\mathbf{u}^{\phi}(t) := e^{\gamma t} S_t \phi, \quad t \in (0,\infty), 
\quad \text{and} \quad \mathbf{u}^{\phi}(0) := \phi \quad \text{in $L^{2}(U)$}.
\]
Then, for each $\phi \in L^2(U)$,  $\mathbf{u}^{\phi} \in C([0,\infty); L^2(U))$ and
\begin{equation} \label{gammal2contra}
\|\mathbf{u}^{\phi}(t) \|_{L^2(U)} \leq e^{\gamma t} \|\phi\|_{L^2(U)} \quad \text{ for all $t \in [0,\infty)$}.
\end{equation}
Moreover, $\mathbf{u}^f \in C([0,\infty); L^2(U))\cap L^2(0,T; H^{1,2}_0(U))$ satisfies $\mathbf{u}^f(0)=f$ in $L^2(U)$ and it follows from \eqref{semigenolid} that
\begin{align*}
&\int_0^T \mathcal{B}_{\gamma}(\mathbf{u}^f(t), \varphi\, \eta(t))\,dt =\int_0^T \mathcal{B}_{\gamma}(e^{\gamma t}S_t f, \varphi\, \eta(t))\,dt \\
&= -\int_0^T \int_{U} \partial_t S_t f \cdot \varphi \, e^{\gamma t}\eta(t)\,dx = \int_0^T \int_{U} S_t f \cdot \varphi \, \partial_t (e^{\gamma t} \eta ) \,dxdt \\
&= \int_0^T \int_{U} \mathbf{u}^f(t) \cdot \varphi \, \partial_t \eta \,dxdt 
+ \gamma \int_0^T \int_{U}\mathbf{u}^f(t)  \cdot \varphi \eta(t)\,dxdt 
\quad \text{for all $\varphi \in H^{1,2}_0(U)$ and $\eta \in C_0^{\infty}((0,T))$}.
\end{align*}
Consequently,
\begin{align} 
&\int_0^T \int_{U} \left\langle A \nabla \mathbf{u}^f(t), \nabla (\varphi \eta) \right\rangle  
+ \left\langle \mathbf{H}, \nabla \mathbf{u}^f(t) \right\rangle (\varphi \eta) + (c+\theta) \mathbf{u}^f(t) (\varphi \eta) \,dxdt  \nonumber \\
&\quad = \int_0^T \int_{U} \mathbf{u}^f(t) \cdot \partial_t (\varphi \eta) \,dxdt \quad \text{ for all $\varphi \in H^{1,2}_0(U)$ and $\eta \in C_0^{\infty}((0,T))$}. \label{weaksolpara}
\end{align}
Moreover, it follows from \eqref{weaksolpara} and Theorem~\ref{equiproweakso} that $\mathbf{u}^f$ is a weak solution to \eqref{bvpara}, where $g$ is replaced by $f$. In particular, replacing $\mathbf{v}$ with $\mathbf{u}^f$ in Theorem~\ref{equiproweakso}(ii), we obtain that
\begin{align}
   &\frac12 \int_0^T \frac{d}{dt} \left(\|\mathbf{u}^f(t)\|_{L^2(U)}^2 \right)\,dt  \nonumber \\
    &\quad + \int_0^T \int_U \langle A \nabla \mathbf{u}^f(t), \nabla \mathbf{u}^f(t) \rangle 
    + \langle \mathbf{H}, \nabla \mathbf{u}^f(t) \rangle \mathbf{u}^f(t) 
    + (c+\theta) \mathbf{u}^f(t) \mathbf{u}^f(t) \,dx \,dt = 0,  \label{pdeweaktim}
\end{align}
and hence, by using the fundamental theorem of calculus and Proposition \ref{stamenest}(ii),
\[
\frac12 \|\mathbf{u}^f(T)\|^2_{L^2(U)} - \frac12 \|f\|_{L^2(U)}^2 
+ \frac{\lambda}{2} \int_0^T \| \nabla \mathbf{u}^f \|^2_{L^2(U)} \,dt 
- \gamma \int_{0}^T \| \mathbf{u}^f \|^2_{L^2(U)} \,dt \leq 0,
\]
where $\gamma := \frac{N^2}{\lambda}$ and $N \geq 0$ is the constant appearing in Proposition \ref{stamenest}(ii).
Hence, it follows from \eqref{pdeweaktim} and \eqref{gammal2contra} that
\begin{align*}
\frac{\lambda}{2} \| \mathbf{u}^f\|^2_{L^2(0,T; H^{1,2}_0(U))} 
&\leq \gamma \int_{0}^T \| \mathbf{u}^f \|^2_{L^2(U)} \,dt + \frac12 \|f\|^2_{L^2(U)} \\
&\leq  \left( \gamma \int_0^T e^{2\gamma t} \,dt + \frac12 \right) \|f\|^2_{L^2(U)} 
= \frac{1}{2}e^{2\gamma T}  \|f\|^2_{L^2(U)},
\end{align*}
which yields
\begin{equation} \label{energyestima}
\| \mathbf{u}^f\|_{L^2(0,T; H^{1,2}_0(U))} 
\leq \left(\frac{e^{2\gamma T} }{\lambda}\right)^{\frac12} \|f\|_{L^2(U)}.
\end{equation}
%%%%%%%
\iffalse
Now, let $g \in L^2(U)$ and define $g_n := n R_n g \in D(\mathcal{L})$ for $n \geq 1$. 
Then, by the strong continuity of the resolvent $(R_{\alpha})_{\alpha>0}$ on $L^2(U)$, we have 
$\lim_{n \rightarrow \infty} g_n = g$ in $L^2(U)$. 
From the $L^2(U)$-contraction property of $(S_t)_{t>0}$, \eqref{energyestima} where $f$ is replaced by $g_n$, and the weak compactness of $L^2(0,T; H^{1,2}_0(U))$, we obtain $\mathbf{u}^g \in L^2(0,T; H^{1,2}_0(U))$ and there exists a subsequence of $(\mathbf{u}^{g_n})_{n \geq 1}$, say again  $(\mathbf{u}^{g_n})_{n \geq 1}$, such that
\begin{equation} \label{weakconver}
\lim_{n \rightarrow \infty} \mathbf{u}^{g_n} = \mathbf{u}^g \quad \text{weakly in } L^2(0,T; H^{1,2}_0(U)),
\end{equation}
and the estimates in \eqref{energyestima} remain valid where $f$ is replaced by $g$.
Noting \eqref{gammal2contra} and applying \eqref{weakconver} to the sequence of weak solutions $(\mathbf{u}^{g_n})_{n \geq 1}$  to \eqref{bvpara}, where $g$ is replaced by $g_n$, we obtain the desired conclusion for (i), where we write $\mathbf{u}^g$ as $\mathbf{u}$. \\ \\
\fi
%%%%%%%%%
Now, let $g \in L^2(U)$ and define $g_n := n R_n g \in D(\mathcal{L})$ for $n \geq 1$. Then, by the strong continuity of the resolvent $(R_{\alpha})_{\alpha>0}$ on $L^2(U)$, we have 
\[
\lim_{n \rightarrow \infty} g_n = g \quad \text{in } L^2(U).
\]
Moreover, since $(S_t)_{t>0}$ is a contraction semigroup,
\[
\|\mathbf{u}^{g_n}-\mathbf{u}^{g}\|_{L^2(0,T;L^2(U))}^2
\leq
\left(\int_0^T e^{2\gamma t}\,dt\right)
\|g_n-g\|_{L^2(U)}^2
\longrightarrow 0.
\]
On the other hand, from \eqref{energyestima} with $f$ replaced by $g_n$ and the weak compactness of $L^2(0,T; H^{1,2}_0(U))$, there exist a subsequence of $(\mathbf{u}^{g_n})_{n \geq 1}$, say again $(\mathbf{u}^{g_n})_{n \geq 1}$, and some $\mathbf{v} \in L^2(0,T;H^{1,2}_0(U))$ such that
\[
\lim_{n \rightarrow \infty} \mathbf{u}^{g_n} = \mathbf{v}
\quad \text{weakly in } L^2(0,T;H^{1,2}_0(U)).
\]
The preceding strong convergence in $L^2(0,T;L^2(U))$ implies that $\mathbf{v}=\mathbf{u}^g$. Hence, $\mathbf{u}^g \in L^2(0,T;H^{1,2}_0(U))$ and
\begin{equation} \label{weakconver}
\lim_{n \rightarrow \infty} \mathbf{u}^{g_n} = \mathbf{u}^g 
\quad \text{weakly in } L^2(0,T; H^{1,2}_0(U)).
\end{equation}
Moreover, the estimate in \eqref{energyestima} remains valid with $f$ replaced by $g$ by weak lower semicontinuity.
Noting \eqref{gammal2contra} and applying \eqref{weakconver} to the sequence of weak solutions $(\mathbf{u}^{g_n})_{n \geq 1}$ to \eqref{bvpara}, where $g$ is replaced by $g_n$, we obtain the desired conclusion for (i), where we write $\mathbf{u}^g$ as $\mathbf{u}$. \\ \\
(ii) Since $(\mathcal{B}_{\gamma}, H^{1,2}_0(U))$ satisfies the strong sector condition, $(S_t)_{t>0}$ is an analytic semigroup by \cite[Chapter I, Corollary 2.21]{MR92}. Therefore, by \cite[Proposition 2.21(iv)]{Lu95}, (ii) follows.
\end{proof}

\begin{rem}
The existence and uniqueness of weak solutions to \eqref{bvpara} were established in Theorems \ref{theouniquene} and \ref{mainwellposth} under assumption {\bf (Hy)}. Notably, the condition {\bf (Hy)} on the drift coefficient $\mathbf{H}$ is slightly more general than the structural condition {\bf (S)} used in our main result, Theorem \ref{maintheore}. However, the estimates in Theorem \ref{mainwellposth}, specifically the $L^2$-estimate \eqref{hcondil2est} and the energy estimate \eqref{hcondienergy}, both exhibit exponential growth on the right-hand side as $T$ increases. In sharp contrast, the energy estimate \eqref{energyrhonnewori} in our main theorem remains uniformly bounded independently of $T$, and the $L^2$-estimate \eqref{expdecaacer} even shows exponential decay as $t$ increases.
Achieving this decay behavior in our framework is critically dependent on the drift coefficient $\mathbf{H}$ satisfying the structural condition {\bf (S)}. Whether exponential decay estimates similar to \eqref{expdecaacer} can be obtained under the more general condition {\bf (Hy)} remains an interesting open problem.
\end{rem}

\section{Weighted Semigroups via Dirichlet forms} \label{dfapproach}
In Section \ref{semigrouponl2dx}, we derived the existence and uniqueness of solutions to the initial-boundary value problems for the parabolic equations \eqref{bvpara}, as well as basic energy and $L^2$ estimates, by identifying the solutions with semigroups on the $L^2$-space with respect to the Lebesgue measure.
However, these solutions provide no information on whether they decay to zero as $t \to \infty$. In fact, the $L^2$ estimates become less useful for large $t$, since the factor $e^{\gamma t}$ grows rapidly.
To overcome this issue, we consider an appropriate weighted measure $\mu = \rho\,dx$ and construct semigroups on the $L^2$-space with respect to $\mu$ via the theory of Dirichlet forms.\\ \\
In this section, we consider the following condition independently:
\\ \\
\textbf{(T)}: $U$ is a bounded open subset of $\mathbb{R}^d$ with $d \geq 2$. $\rho \in H^{1,2}(U) \cap L^{\infty}(U)$ and there exists a constant $\beta>0$ such that 
$$
\rho(x) \geq \beta \quad \text{ for a.e. $x \in U$}.
$$
$\mu=\rho\,dx$ and
$\mathbf{B} \in L^p(U, \mathbb{R}^d)$ with $p \in (d, \infty)$ satisfies
$$
\int_{U} \langle \mathbf{B}, \nabla f \rangle  \,d\mu=0 \quad \text{ for all $f \in C_0^{\infty}(U)$}.
$$ 
$A = (a_{ij})_{1 \leq i,j \leq d}$ is a \textit{(possibly non-symmetric) matrix of measurable functions on $\mathbb{R}^d$} such that for some constants $M > 0$ and $\lambda > 0$, \eqref{ellipticity} holds.
%\[
%\max_{1 \leq i,j \leq d} |a_{ij}(x)| \leq M, \quad \langle A(x)\xi, \xi \rangle \geq \lambda \|\xi\|^2 \quad \text{for a.e. } x \in \mathbb{R}^d \text{ and for all } \xi \in \mathbb{R}^d.
%\]
$c \in L^{s}(U)$ satisfies $c \geq 0$ a.e. in $U$,  where $s \in (1, \infty)$ if $d=2$, and $s:=\frac{d}{2}$ if $d \geq 3$.\\ \\
Here and below, $\inf_U\rho$ and $\sup_U\rho$ denote the essential infimum and essential supremum, respectively. Under assumption {\bf (T)}, define the bilinear form $(\mathcal{E}, C_0^{\infty}(U))$ by
\begin{equation} \label{underlydf}
\mathcal{E}(f,g) = \int_U \langle A \nabla f, \nabla g \rangle \, d\mu 
+ \int_U \langle \mathbf{B}, \nabla f \rangle g \, d\mu + \int_U c f g \,d\mu, \quad f, g \in C_0^\infty(U).
\end{equation}
Then, using the method in \cite[Theorem 3.2(i), (ii)]{L25jm},
$(\mathcal{E}, C_0^{\infty}(U))$ satisfies the strong sector condition and is closable on $L^2(U, \mu)$ (cf. \cite[Chapter I, Section 2]{MR92}). Moreover, by \cite[Theorem 3.2(iv)]{L25jm}
the closure of $(\mathcal{E}, C_0^{\infty}(U))$ on $L^2(U, \mu)$ denoted by $(\mathcal{E}, D(\mathcal{E}))$ is a Dirichlet form (cf. \cite[Chapter I, Section 4]{MR92}). In particular, by \cite[Theorem 3.2(iii)]{L25jm}, $D(\mathcal{E})=H^{1,2}_0(U)$.
A particularly important example of such a vector field will be constructed in Theorem~\ref{leetransform}. More precisely, it is obtained from the original drift coefficient $\mathbf{H}$ through the divergence-free transformation introduced there.
\begin{theo} \label{weaksolsemi}
Assume that {\bf (T)} holds. Let $(\mathcal{E}, C_0^{\infty}(U))$ be the bilinear form defined as in \eqref{underlydf}, and let $(\mathcal{E}, D(\mathcal{E}))$ denote its closure, which is a Dirichlet form on $L^2(U, \mu)$. Let $(T_t)_{t>0}$ be a sub-Markovian $C_0$-semigroup of contractions on $L^2(U, \mu)$ associated to $(\mathcal{E}, D(\mathcal{E}))$. Let $\theta \in [0, \infty)$ be a constant. For each $\phi \in L^2(U)$, define $\mathbf{u}^\phi \in C([0, \infty); L^2(U))$ by
\begin{align*}
\mathbf{u}^{\phi}(t):=e^{-\theta t} T_t \phi, \quad t \in (0, \infty), \qquad \mathbf{u}^{\phi}(0):=\phi.
\end{align*}
Let $g \in L^2(U)$ and $T \in (0, \infty)$.
Then, the following hold:
\begin{itemize}
\item[(i)]
The function $\mathbf{u}^{g} \in C([0, \infty); L^2(U, \mu)) \cap L^2(0,T; H^{1,2}_0(U))$ satisfies, for all $v \in H^{1,2}_0(U)$ and $\eta \in C_0^{\infty}((0,T))$,
\[
\int_0^T \int_{U}\mathbf{u}^g(t) \partial_t (v \eta) \, d\mu dt= \int_0^T \int_U \langle A \nabla \mathbf{u}^g, \nabla (v \eta) \rangle 
+ \langle \mathbf{B}, \nabla \mathbf{u}^g \rangle (v \eta) + (c+\theta) \mathbf{u}^g (v \eta) \,d\mu \,dt.
\]
In particular,
$$
\int_0^T \int_{U}  \|\nabla \mathbf{u}^{g}(t) \|^2 d\mu dt  \leq \frac{1}{2\lambda} \| g\|^2_{L^2(U, \mu)}.
$$
\item[(ii)]
Let 
$$
\kappa_0:= \theta+ \frac{d^2 \lambda }{4(d-1)^2 |U|^{\frac{2}{d}}} \left(\frac{\inf_{U} \rho}{\sup_{U} \rho} \right). 
$$
Then, $\mathbf{u}^g$ satisfies
$$
\|\mathbf{u}^{g}(t)\|_{L^2(U, \mu)} \leq e^{-\kappa_0 t} \|g\|_{L^2(U, \mu)} \quad \text{ for all $t \in (0, \infty)$}.
$$
\item[(iii)]
Let $r \in [1, \infty)$. If $g \in L^2(U, \mu) \cap L^r(U, \mu)$, then $\mathbf{u}^{g} \in C([0, \infty); L^r(U, \mu))$ and
$$
\| \mathbf{u}^{g}(t) \|_{L^r(U, \mu)} \leq e^{-\theta t} \|g \|_{L^r(U, \mu)} \quad \text{ for all $t \in (0, \infty)$}.
$$
Moreover, if $g \in L^{\infty}(U)$, then
$$
\| \mathbf{u}^{g}(t) \|_{L^{\infty}(U, \mu)} \leq e^{-\theta t} \|g \|_{L^{\infty}(U, \mu)} \quad \text{ for all $t \in (0, \infty)$}.
$$
\end{itemize}
\end{theo}
\begin{proof}
(i)
Let $(L, D(L))$ denote the generator, and $(G_{\alpha})_{\alpha > 0}$ the sub-Markovian $C_0$-resolvent of contractions on $L^2(U, \mu)$ associated with the Dirichlet form $(\mathcal{E}, D(\mathcal{E}))$.  Fix $g \in L^2(U, \mu)$ and define $g_n := nG_n g \in D(L)$ for each $n \geq 1$. Then, by the strong continuity of $(G_{\alpha})_{\alpha>0}$ on $L^2(U, \mu)$, $\lim_{n \rightarrow \infty} g_n =g$ in $L^2(U, \mu)$.
In particular, by the $L^2(U, \mu)$-contraction property of $(G_{\alpha})_{\alpha>0}$, 
\begin{equation} \label{contraestimgn}
\|g_n\|_{L^2(U, \mu)} \leq \|g\|_{L^2(U, \mu)} \quad \text{ for all $n \geq 1$}.
\end{equation}
Moreover, for each $t \in (0, \infty)$, we have $T_t g_n \in D(L) \subset D(\mathcal{E}) = H^{1,2}_0(U)$,
\begin{align*}
\partial_t \mathbf{u}^{g_n}(t) =\partial_t(e^{-\theta t} T_t g_n)= -\theta e^{-\theta t} T_t g_n +e^{-\theta t} LT_t g_n \quad \text{ in $L^2(U, \mu)$}.
\end{align*}
Thus, for each $t \in (0,\infty)$ and $v \in H^{1,2}_0(U)$
\begin{align}
& -\int_U \partial_t \mathbf{u}^{g_n}(t) v d\mu  = \int_U \theta e^{-\theta t} T_t g_n \cdot v \,d\mu -  \int_U e^{-\theta t} LT_t g_n \cdot v \,d\mu  \nonumber \\
&\quad =  \int_U \theta e^{-\theta t} T_t g_n \cdot v\,d\mu + e^{-\theta t}  \mathcal{E}(T_t g_n,  v ) \nonumber \\
&\quad = \theta \int_U \mathbf{u}^{g_n}(t) \cdot v d\mu + \int_U \langle A \nabla \mathbf{u}^{g_n}(t), \nabla v \rangle \, d\mu 
+ \int_U \langle \mathbf{B}, \nabla \mathbf{u}^{g_n}(t) \rangle v \, d\mu + \int_U c \mathbf{u}^{g_n}(t) v \,d\mu. \label{intpartfor}
\end{align}
Meanwhile, replacing $v$ by $\mathbf{u}^{g_n}(t)$ in \eqref{intpartfor}, it follows that 
\begin{align} \label{grownineq}
-\frac12 \partial_t \left( \int_{U} |\mathbf{u}^{g_n}(t)|^2 d\mu \right) \geq \theta \int_{U} |\mathbf{u}^{g_n}|^2 d\mu + \lambda \int_{U} \| \nabla \mathbf{u}^{g_n}(t) \|^2\,d\mu,
\end{align}
and hence by \eqref{contraestimgn}
\begin{align}
\lambda \int_0^T \int_{U}  \|\nabla \mathbf{u}^{g_n}(t) \|^2 d\mu dt &\leq \frac12 \int_{U} |g_n|^2 d\mu -\frac12 \int_{U} |\mathbf{u}^{g_n}(T)|^2 \,d\mu  \leq \frac12\| g_n\|^2_{L^2(U, \mu)} \leq \frac12 \|g\|^2_{L^2(U, \mu)}.  \label{energyestimu}
\end{align}
%Moreover,
%\begin{align*}
%\mathcal{E}(\mathbf{u}^{g_n}(t), \mathbf{u}^{g_n}(t)) &= \mathcal{E}(T_t g_n,  T_t g_n ) = -\int_{U} L T_t g_n \cdot T_t g_n d\mu \\
%&\leq \| L g_n \|_{L^2(U, \mu)} \| g_n \|_{L^2(U, \mu)} \quad \text{ for all $t \in (0, \infty)$}.
%\end{align*}
Therefore, $\mathbf{u}^{g_n} \in L^2(0,T; H^{1,2}_0(U))$, and for each $v \in H^{1,2}_0(U)$ and $\eta \in C_0^{\infty}((0,T))$,
\begin{align}
&\quad \int_0^T \int_U  \mathbf{u}^{g_n}(t) \partial_t(v\eta) d\mu dt= -\int_0^T \int_U \partial_t \mathbf{u}^{g_n}(t) (v\eta) d\mu dt \nonumber \\
&= \int_0^T \int_U \langle A \nabla \mathbf{u}^{g_n}(t), \nabla (v\eta) \rangle \, d\mu dt +\int_0^T \int_U \langle \mathbf{B}, \nabla \mathbf{u}^{g_n}(t) \rangle (v\eta) \, d\mu dt +\int_0^T \int_U (c+\theta) \mathbf{u}^{g_n}(t) (v\eta) \,d\mu dt. \label{estimiden}
\end{align}
By the $L^2(U,\mu)$-contraction property of $(T_t)_{t>0}$, we have
\[
\|\mathbf{u}^{g_n}-\mathbf{u}^{g}\|_{L^2(0,T;L^2(U,\mu))}^2 \leq \left(\int_0^T e^{-2\theta t}\,dt\right) \|g_n-g\|_{L^2(U,\mu)}^2
\longrightarrow 0 \quad \text{as } n\rightarrow\infty.
\]
Moreover, by the energy estimate \eqref{energyestimu} and the weak compactness of
$L^2(0,T;H^{1,2}_0(U))$, there exists a subsequence of $(\mathbf{u}^{g_n})_{n\geq1}$, still denoted by
$(\mathbf{u}^{g_n})_{n\geq1}$, which converges weakly in $L^2(0,T;H^{1,2}_0(U))$. The above strong convergence in $L^2(0,T;L^2(U,\mu))$ shows that its weak limit must coincide with $\mathbf{u}^{g}$. Hence,
\begin{equation} \label{weakconvesol}
\lim_{n\rightarrow\infty}\mathbf{u}^{g_n}=\mathbf{u}^{g}, \quad \text{weakly in $L^2(0,T;H^{1,2}_0(U))$}.
\end{equation}
Therefore, the assertion (i) follows from \eqref{weakconvesol}, \eqref{estimiden} and \eqref{energyestimu}. \\ \\
(ii) 
\eqref{grownineq} and \eqref{sobolhold} induce that for each $t \in (0, \infty)$
\begin{align} \label{grownineq2}
&-\frac12 \partial_t \left( \int_{U} |\mathbf{u}^{g_n}(t)|^2 d\mu \right) \geq \theta \int_{U} |\mathbf{u}^{g_n}|^2 d\mu + \lambda \int_{U} \| \nabla \mathbf{u}^{g_n}(t) \|^2\,d\mu  \nonumber \\
& \geq   \theta \int_{U} |\mathbf{u}^{g_n}|^2 d\mu + \lambda (\inf_{U} \rho) \int_{U} \| \nabla \mathbf{u}^{g_n}(t) \|^2\,dx  \nonumber \\
& \geq     \kappa_0 \int_{U} |\mathbf{u}^{g_n}|^2 d\mu, \nonumber
\end{align}
where $\kappa_0:= \theta+ \frac{d^2 \lambda }{4(d-1)^2 |U|^{\frac{2}{d}}} \left(\frac{\inf_{U} \rho}{\sup_{U} \rho} \right)$. Define $\Psi_n(t) := \int_U |\mathbf{u}^{g_n}(t)|^2 \, d\mu$ for $t \in (0, \infty)$.
Then $\frac{d}{dt}\left(e^{2\kappa_0t}\Psi_n(t) \right) \leq 0$ for all $t \in (0, \infty)$, and hence by \eqref{contraestimgn}
$$
0 \leq e^{2\kappa_0 t} \Psi_n(t) \leq \Psi_n(0)=\int_U |g_n|^2 d\mu \leq \int_U |g|^2 d\mu.
$$
Thus,
$$
\|\mathbf{u}^{g_n}(t)\|_{L^2(U, \mu)} \leq e^{-\kappa_0 t} \|g\|_{L^2(U, \mu)} \quad \text{ for all $t \in (0, \infty)$}.
$$
Letting $n \rightarrow \infty$, the conclusion follows by the strong continuity of $(T_t)_{t>0}$ on $L^2(U, \mu)$. \\ \\
(iii) We claim that the restriction $(T_t)_{t > 0}|_{L^\infty(U, \mu)}$ extends to a sub-Markovian $C_0$-semigroup of contractions on $L^r(U, \mu)$ for all $r \in [1, \infty)$. Since $(\mathcal{E}, D(\mathcal{E}))$ is a Dirichlet form, the adjoint semigroup $(T_t')_{t > 0}$ is also a sub-Markovian $C_0$-semigroup of contractions on $L^2(U, \mu)$. Then, for any $f \in L^\infty(U, \mu)$ and $t > 0$, we have
\begin{equation} \label{l1contra}
\int_U |T_t f|  d\mu \leq \int_U T_t |f| \cdot 1_U  d\mu = \int_U |f| \cdot T_t' 1_U  d\mu \leq \int_U |f|  d\mu.
\end{equation}
Thus, $(T_t)|_{L^{\infty}(U, \mu)}$ extends to a sub-Markovian contraction semigroup on $L^1(U, \mu)$, say again $(T_t)_{t>0}$.
By the inequality \eqref{l1contra}, the sub-Markovian property of $(T_t)_{t > 0}$, and the Riesz–Thorin interpolation theorem (see \cite[15, Chapter 2, Theorem 2.1]{SS11}), it follows that $(T_t)_{t>0}$ is a sub-Markovian contraction semigroup on $L^r(U, \mu)$
for each $r \in [1, \infty)$. In particular, for each $f \in L^{\infty}(U, \mu)$
$$
\|T_t f - f \|_{L^1(U, \mu)} \leq \mu(U)^{1/2} \| T_t f -f \|_{L^2(U, \mu)} \longrightarrow 0 \quad  \text{ as $t \rightarrow 0+$}.
$$
Hence, by the $L^1(U, \mu)$-contraction property of $(T_t)_{t>0}$ and the standard $3$-$\varepsilon$ argument, $(T_t)_{t>0}$ is strongly continuous on $L^1(U, \mu)$. Moreover, for each $r \in [1, \infty)$ and $f \in L^{\infty}(U, \mu)$,
$$
\|T_t f - f\|^r_{L^r(U, \mu)} = \int_{U} |T_t f - f| |T_t f - f|^{r-1} d\mu \leq 2^{r-1}\|f\|_{L^{\infty}(U, \mu)}^{r-1}  \|T_t f - f \|_{L^1(U, \mu)} \longrightarrow 0 \quad  \text{ as $t \rightarrow 0+$}.
$$
Therefore, again by using $L^r(U, \mu)$-contraction properties of $(T_t)_{t>0}$ and the standard $3$-$\varepsilon$ argument, $(T_t)_{t>0}$ is strongly continuous on $L^r(U, \mu)$ for each $r \in [1, \infty)$, so the claim is shown. Finally, since $\mathbf{u}^g(t)=e^{-\theta t}T_t g$ for all $t>0$, the conclusion in (iii) follows.
\end{proof}

\section{Identification of the weighted semigroups} \label{mainsection}
In Section \ref{dfapproach}, we constructed a semigroup using appropriate weighted measures $\mu$ and divergence-free vector fields with respect to $\mu$, through the associated Dirichlet form.
A particularly notable feature was that the constant appearing in the $L^2$-contraction estimates takes the form $e^{-\kappa_0 t}$, which exhibits exponential decay as time grows.
In this section, we return to our original problem, the initial-boundary value problem \eqref{bvpara}, and rigorously demonstrate how the sub-Markovian $C_0$-semigroup of contractions $(T_t)_{t > 0}$ on $L^2(U, \mu)$, constructed in Section \ref{dfapproach}, can be identified as its solution. To achieve this, we develop the argument based on the idea introduced in the recent result \cite{L25jm}, which provides a novel approach to linear elliptic equations.\\ 
In this section, we consider condition {\bf (S)} in the introduction which is also assumed in \cite{L25jm}:
It is straightforward to see that if {\bf (S)} holds, $\theta \in [0, \infty)$ is a constant and $c \in L^{s}(U)$ satisfies $c \geq 0$ a.e. in $U$, where $s \in (1, \infty)$ if $d = 2$, and $s := \frac{d}{2}$ if $d \geq 3$, then {\bf (Hy)} follows. We now introduce a key result stated in \cite{L25jm}, which ensures the existence of $\rho$ under the assumption {\bf (S)}.

\begin{theo} \label{existinvar}
\label{rho_existence}
Assume that \textbf{(S)} holds. Then, the following hold:
\begin{itemize}
    \item[(i)] Let $x_1 \in U$. Then, there exists $\rho \in H^{1,2}(B_{4R}(x_0)) \cap C(B_{4R}(x_0))$ with $\rho(x) > 0$ for all $x \in B_{4R}(x_0)$ and $\rho(x_1) = 1$ such that
    \[
        \int_{B_{4R}(x_0)} \langle A^T \nabla \rho + \rho \mathbf{H}, \nabla \varphi \rangle dx = 0
        \quad \text{for all } \varphi \in C_0^\infty(B_{4R}(x_0)).
    \]
    
    \item[(ii)] Let $\rho$ be as in Theorem \ref{rho_existence}(i). Then, there exists a constant $K_1 \geq 1$ which only depends on $d, \lambda, M, R, p, \|h\|_{L^p(U)}$ such that
    \[
        \max_{\overline{B}_{3R}(x_0)} \rho \leq K_1 \min_{\overline{B}_{3R}(x_0)} \rho,
        \quad \text{and hence} \quad
        1 \leq \frac{\max_{\overline{U}} \rho}{\min_{\overline{U}} \rho} \leq K_1.
    \]
\item[(iii)] If $a_{ij} \in C(\mathbb{R}^d)$ for all $1 \leq i,j \leq d$, then $\rho \in H^{1,p}(U)$.
\end{itemize}
\end{theo}
\begin{proof}
(i) and (ii) follow from \cite[Theorem 3.1(i), (ii)]{L25jm}, and (iii) follows from \cite[Theorem 4.1]{L25jm}.
\end{proof}

\begin{theo} \label{leetransform}
 Assume that {\bf (S)} holds and that $a_{ij} \in C(\mathbb{R}^d)$ for all $1 \leq i,j \leq d$.
Let $\rho \in H^{1,p}(U) \cap C(\overline{U})$ be a strictly positive function on $\overline{U}$ constructed as in Theorem \ref{rho_existence}. Define $\mu:=\rho \,dx$ and
\[
\mathbf{B} := \mathbf{H} + \frac{1}{\rho} A^T \nabla \rho \quad \text{on } U.
\]
Let $\theta \in [0, \infty)$ be a constant and  let $c \in L^s(U)$ satisfy $c \geq 0$ a.e. in $U$,  where $s \in (1, \infty)$ if $d=2$, and $s:=\frac{d}{2}$ if $d \geq 3$.
Then, $\mathbf{B} \in L^p(U, \mathbb{R}^d, \mu)$ and
\begin{equation} \label{divfreecond}
\int_U \langle  \mathbf{B}, \nabla \varphi \rangle d\mu = 0 \quad \text{for all } \varphi \in C_0^\infty(U).
\end{equation}
Let $\mathbf{u} \in L^2(0,T; H^{1,2}_0(U))$ and $\eta \in C_0^{\infty}((0,T))$.
Then, the following (i) and (ii) are equivalent:
\begin{itemize}
\item[(i)]
\begin{equation*} 
\int_0^T \int_{U} \mathbf{u}(t) \partial_t (\psi \eta)\,dxdt =\int_0^T \int_U \langle A \nabla \mathbf{u}(t), \nabla (\psi \eta) \rangle 
+ \langle \mathbf{H}, \nabla \mathbf{u}(t) \rangle (\psi \eta) + (c+\theta) \mathbf{u}(t) (\psi \eta) \,dx \,dt \quad \text{for all $\psi \in H^{1,2}_0(U)$}.
\end{equation*}
\item[(ii)]
\[
\int_0^T \int_{U}\mathbf{u} (t) \partial_t (\varphi  \eta) \, d\mu dt= \int_0^T \int_U \langle A \nabla \mathbf{u}(t), \nabla (\varphi  \eta) \rangle 
+ \langle \mathbf{B}, \nabla \mathbf{u}(t) \rangle (\varphi  \eta) + (c+\theta) \mathbf{u}(t) (\varphi \eta) \,d\mu \,dt \quad \text{for all $\varphi \in H^{1,2}_0(U)$}.
\]
\end{itemize}
\end{theo}
\begin{proof}
Since $\mathbf{H} \in L^p(U, \mathbb{R}^d)$ and $\rho$ is bounded below and above by strictly positive constants,  $\mathbf{B} \in L^p(U, \mathbb{R}^d, \mu)$.
(i) $\Rightarrow$ (ii). Let $\varphi \in H^{1,2}_0(U)$ be arbitrarily given. Then, replacing $\psi$ by $\rho \varphi \in H^{1,2}_0(U)$ in (i),
\begin{align*}
&\int_0^T \int_{U}\mathbf{u} (t) \partial_t (\varphi  \eta) \, d\mu dt=\int_0^T \left(\int_U \langle A \nabla \mathbf{u}(t), \nabla (\rho \varphi) \rangle 
+ \langle \mathbf{H}, \nabla \mathbf{u}(t) \rangle (\rho \varphi) + (c+\theta) \mathbf{u}(t) (\rho \varphi) \,dx \right) \eta(t) \,dt  \\
&=\int_0^T \left(\int_U \langle A \nabla \mathbf{u}(t), \nabla \varphi \rangle\rho 
+ \langle \rho \mathbf{H} + A^T \nabla \rho, \nabla \mathbf{u}(t) \rangle \varphi + (c+\theta) \mathbf{u}(t) (\rho \varphi) \,dx \right) \eta(t) \,dt  \\
&=\int_0^T \left(\int_U \langle A \nabla \mathbf{u}(t), \nabla \varphi \rangle 
+ \langle \mathbf{B}, \nabla \mathbf{u}(t) \rangle \varphi + (c+\theta) \mathbf{u}(t) \varphi \,d\mu \right) \eta(t) \,dt \\
&= \int_0^T \int_U \langle A \nabla \mathbf{u}(t), \nabla (\varphi  \eta) \rangle 
+ \langle \mathbf{B}, \nabla \mathbf{u}(t) \rangle (\varphi  \eta) + (c+\theta) \mathbf{u}(t) (\varphi \eta) \,d\mu \,dt,
\end{align*}
and hence (ii) follows. \\ \\
(ii) $\Rightarrow$ (i). Let $\psi \in H^{1,2}_0(U)$ be arbitrarily given. Then, replacing $\varphi$ by $\frac{\psi}{\rho} \in H^{1,2}_0(U)$ in (ii),
\begin{align*}
&\int_0^T \int_{U} \mathbf{u}(t) \partial_t (\psi \eta)\,dxdt =\int_0^T \int_U \left \langle \rho A \nabla \mathbf{u}(t), \nabla \left(\frac{\psi}{\rho}  \eta \right) \right \rangle 
+ \langle \mathbf{B}, \nabla \mathbf{u}(t) \rangle (\psi  \eta) + (c+\theta) \mathbf{u}(t) (\psi \eta) \,dx \,dt \\
&= \int_0^T \left( \int_U \left \langle \rho A \nabla \mathbf{u}(t), \nabla \left(\frac{\psi}{\rho} \right) \right \rangle 
+ \langle \mathbf{B}, \nabla \mathbf{u}(t) \rangle \psi + (c+\theta) \mathbf{u}(t) \psi \,dx \right) \eta(t) \,dt  \\
&=\int_0^T \left( \int_U	\langle A \nabla \mathbf{u}, \nabla \psi	 \rangle  +\left \langle  -\frac{1}{\rho}A^T \nabla \rho+\mathbf{B}, \nabla \mathbf{u}(t) \right \rangle \psi + (c+\theta) \mathbf{u}(t)\psi \,dx \right) \eta(t)\,dt \\
&=\int_0^T \left( \int_U	\langle A \nabla \mathbf{u}, \nabla \psi	 \rangle  +\left \langle  \mathbf{H}, \nabla \mathbf{u}(t) \right \rangle\psi + (c+\theta) \mathbf{u}(t) \psi \,dx\right) \eta(t)\,dt \\
&=\int_0^T \int_U \langle A \nabla \mathbf{u}(t), \nabla (\psi  \eta) \rangle 
+ \langle \mathbf{H}, \nabla \mathbf{u}(t) \rangle (\psi  \eta) + (c+\theta) \mathbf{u}(t) (\psi \eta) \,dx \,dt,
\end{align*}
and hence (i) follows. 
\end{proof}
\begin{rem}[Divergence-free transformation] \label{leetransformrem}
Theorem~\ref{existinvar} provides a remarkable structural insight: it enables us to transform linear equations with general $L^p$-drifts, where $p \in (d, \infty)$, into equivalent equations with divergence-free drifts under an appropriately weighted measure. In particular, for general drifts $\mathbf{H}$, establishing the coercivity of the associated bilinear form typically requires adding a compensating constant, such as $\frac{N^2}{\lambda}$ in Proposition \ref{stamenest}. However, under the divergence-free structure obtained through Theorem \ref{leetransform}, the corresponding Dirichlet form can be directly constructed without adding such a compensating constant.\\
This transformation between the formulations in parts (i) and (ii) of Theorem \ref{leetransform} not only ensures the existence and uniqueness of solutions, but also facilitates deeper analysis of their regularity and stability. The conceptual foundation of this approach was first introduced in the elliptic setting in~\cite{L25jm}, where it led to new results on existence, uniqueness, and regularity that had previously been inaccessible via classical techniques. We shall refer to this transformation, which preserves the solution while reformulating the drift coefficient into a divergence-free structure, as the divergence-free transformation.
\end{rem}

\begin{lem} \label{weckcol2lem}
Let $T \in (0, \infty)$ and $\zeta:[0, T] \rightarrow [0, \infty)$ be a continuous function. Let $(\mathbf{u}_n)_{n \geq 1} \subset C([0,T]; L^2(U))$ be such that
\begin{equation} \label{l2boundsweakc}
\|\mathbf{u}_n(t) \|_{L^2(U)} \leq \zeta(t) \quad \text{ for all $t \in [0, T]$}.
\end{equation}
Then, there exists $\hat{\mathbf{u}} \in L^2(0,T;L^2(U))$ and a subsequence of $(\mathbf{u}_n)_{n \geq 1}$, say again, $(\mathbf{u}_n)_{n \geq 1}$ such that 
$$
\lim_{n \rightarrow \infty}\mathbf{u}_n = \hat{\mathbf{u}} \quad \text{ weakly in $L^2(0,T;L^2(U))$},
$$
and that
\begin{equation} \label{aeptwiselinqe}
\|\hat{\mathbf{u}}(t)\|_{L^2(U)} \leq \zeta(t) \quad \text{ for a.e. $t \in (0,T)$}.
\end{equation}
\end{lem}
\begin{proof}
By \eqref{l2boundsweakc} and the weak compactness of $L^2(0,T; L^2(U))$, there exist $\hat{\mathbf{u}} \in L^2(0,T;L^2(U))$ and a subsequence of $(\mathbf{u}_n)_{n \geq 1}$, say again $(\mathbf{u}_n)_{n \geq 1}$, such that
\[
\lim_{n \rightarrow \infty} \mathbf{u}_n = \hat{\mathbf{u}} \quad \text{weakly in } L^2(0,T;L^2(U)).
\]
For each $\eta \in C_0^{\infty}((0,T))$ with $\eta \geq 0$ in $(0,T)$, by the Cauchy-Schwarz inequality we have
\begin{align*}
\int_0^T \int_U \mathbf{u}_n(t) \hat{\mathbf{u}}(t) \eta(t) \,dx\,dt 
&\leq \left(\int_0^T \int_U |\mathbf{u}_n(t)|^2 \eta(t) \,dx\,dt \right)^{1/2} \left( \int_0^T \int_U |\hat{\mathbf{u}}(t)|^2 \eta(t) \,dx\,dt \right)^{1/2} \\
&\leq \left(\int_0^T \zeta(t)^2 \eta(t) \,dt \right)^{1/2} \left( \int_0^T \int_U |\hat{\mathbf{u}}(t)|^2 \eta(t) \,dx\,dt \right)^{1/2}.
\end{align*}
Letting $n \rightarrow \infty$, the weak convergence implies that 
$$
\int_0^T \|\hat{\mathbf{u}}(t)\|_{L^2(U)}^2 \, \eta(t) \,dt \leq \left(\int_0^T \zeta(t)^2 \eta(t) \,dt \right)^{1/2} \left( \int_0^T \|\hat{\mathbf{u}}(t)\|_{L^2(U)}^2 \eta(t) \,dt \right)^{1/2}.
$$
Dividing by $\left( \int_0^T \|\hat{\mathbf{u}}(t)\|_{L^2(U)}^2 \eta(t) \,dt \right)^{1/2}$ and squaring both sides yields
$$
\int_0^T \|\hat{\mathbf{u}}(t)\|_{L^2(U)}^2 \eta(t) \,dt \leq \int_0^T \zeta(t)^2 \eta(t) \,dt.
$$
Since this holds for every positive function $\eta \in C_0^\infty((0,T))$, we deduce that $\|\hat{\mathbf{u}}(t)\|_{L^2(U)}^2 \leq \zeta(t)^2$ for a.e. $t \in (0,T)$.
i.e., \eqref{aeptwiselinqe} holds.
\end{proof}
\centerline{}
\noindent
Now we present the proof of our main result in this paper. \\ \\
\noindent\textbf{Proof of Theorem~1.1.}\;
(i) 
{\sf \underline{Step 1)} Construction of an approximation $(\mathbf{u}_n)_{n \geq 1}$ to the unique solution $\mathbf{u}$}:\\
By mollification, let $(A_n)_{n \geq 1} =\big( (a_{ij,n})_{1 \leq i,j \leq d}\big)_{n \geq 1}$ be a sequence of smooth matrix-valued functions satisfying
\[
\max_{1 \leq i,j \leq d} |a_{ij,n}(x)| \leq M, \quad \langle A_n(x) \xi, \xi \rangle \geq \lambda \|\xi\|^2 \quad \text{for all } x \in \mathbb{R}^d \text{ and  } \xi \in \mathbb{R}^d,
\]
and such that for each $1 \leq i,j \leq d$,
\[
\lim_{n \to \infty} a_{ij,n}(x) = a_{ij}(x) \quad \text{for a.e. } x \in \mathbb{R}^d.
\]
Let $\mathbf{u}_n \in L^2(0,T; H_0^{1,2}(U)) \cap C([0,\infty); L^2(U))$ denote the unique weak solution to \eqref{bvpara} with $A$ replaced by $A_n$, as constructed in Theorem~\ref{mainwellposth}. Then $\mathbf{u}_n(0) = g$ for all $n \geq 1$. By Theorem \ref{equiproweakso}, for all $v \in H_0^{1,2}(U)$ and $\eta \in C_0^\infty((0,T))$, the following identity holds:
\begin{equation} \label{iterativeint2}
\int_0^T \int_U -\mathbf{u}_n(t) \partial_t (v \eta) + \langle A_n \nabla \mathbf{u}_n(t), \nabla (v \eta) \rangle 
+ \langle \mathbf{H}, \nabla \mathbf{u}_n(t) \rangle (v \eta) + (c + \theta) \mathbf{u}_n(t) (v \eta) \, dx \, dt = 0.
\end{equation}
Moreover, again by Theorem \ref{mainwellposth}(i), the following estimate holds:
%\begin{equation} \label{gammcontra}
%\|\mathbf{u}_n(t)\|_{L^2(U)} \leq e^{\gamma t} \|g\|_{L^2(U)} \quad \text{for all } t \in [0, \infty),
%\end{equation}
%and
\begin{equation} \label{energweakco}
\| \mathbf{u}_n \|_{L^2(0,T; H_0^{1,2}(U))} \leq \left( \frac{e^{2\gamma T}}{\lambda} \right)^{1/2} \|g\|_{L^2(U)},
\end{equation}
where $\gamma := \frac{N^2}{\lambda}$ and $N \geq 0$ is the constant appearing in Proposition \ref{stamenest}\textup{(ii)}.
%By \eqref{gammcontra} and Lemma \ref{weckcol2lem},  there exist $\hat{\mathbf{u}} \in L^2(0,T;L^2(U))$ and a subsequence of $(\mathbf{u}_n)_{n \geq 1}$, again denoted by $(\mathbf{u}_n)_{n \geq 1}$, such that
%\begin{equation} \label{weakhatu}
%\lim_{n \rightarrow \infty} \mathbf{u}_n = \hat{\mathbf{u}} \quad \text{weakly in } L^2(0,T;L^2(U)),
%\end{equation}
%\begin{equation} \label{l2contracti}
%\|\hat{\mathbf{u}}(t)\|_{L^2(U)} \leq e^{\gamma t} \|g\|_{L^2(U)} \quad \text{ for a.e. $t \in (0,T)$}.
%\end{equation}
%and $\hat{\mathbf{u}}(0) = g$ in $L^2(U)$. %Moreover, by \eqref{gammcontra} and \eqref{weakhatu},
%\[
%\lim_{n \rightarrow \infty}\mathbf{u}_n =\hat{\mathbf{u}} \quad \text{weakly in } L^2(0,T; L^2(U)).
%\]
By \eqref{energweakco} and the weak compactness of $L^2(0,T; H^{1,2}_0(U))$, there exists $\tilde{\mathbf{u}} \in L^2(0,T; H^{1,2}_0(U))$ and a subsequence of $(\mathbf{u}_n)_{n \geq 1}$, again denoted by $(\mathbf{u}_n)_{n \geq 1}$, such that
\begin{equation} \label{weaklysobof}
\lim_{n \rightarrow \infty} \mathbf{u}_n = \tilde{\mathbf{u}} \quad \text{weakly in } L^2(0,T; H^{1,2}_0(U)).
\end{equation}
%It then follows form \eqref{energweakco} that
%\begin{equation} \label{energweakori}
%\| \tilde{\mathbf{u}} \|_{L^2(0,T; H^{1,2}_0(U))} \leq \left( \frac{e^{2\gamma T}}{\lambda} \right)^{1/2} \|g\|_{L^2(U)},
%\end{equation}
%and 
%\begin{equation} \label{uniquehatae}
%\hat{\mathbf{u}} = \tilde{\mathbf{u}} \quad \text{ in $L^2(0,T; L^2(U))$}. 
%\end{equation}
Passing to the limit in \eqref{iterativeint2} by the weak convergence \eqref{weaklysobof}, we obtain that for all $v \in H^{1,2}_0(U)$ and $\eta \in C_0^\infty((0,T))$,
\begin{equation} \label{iterativeintlim}
\int_0^T \int_U -\tilde{\mathbf{u}}(t) \partial_t (v \eta) + \langle A \nabla \tilde{\mathbf{u}}(t), \nabla (v \eta) \rangle 
+ \langle \mathbf{H}, \nabla \tilde{\mathbf{u}}(t) \rangle (v \eta) + (c + \theta) \tilde{\mathbf{u}}(t) (v \eta) \, dx \, dt = 0.
\end{equation}
For a.e. $t \in (0, T)$, define $\tilde{\mathbf{w}}(t) \in H^{-1,2}(U)$ as in \eqref{weaktimedri}. Then, as in the proof of Theorem \ref{equiproweakso} ((iii) $\Rightarrow$ (i)), we have $\tilde{\mathbf{u}}' = \tilde{\mathbf{w}} \in L^2(0,T; H^{-1,2}(U))$. Thus, by \cite[Section 5.9, Theorem 3(i)]{E10}, $\tilde{\mathbf{u}}$ admits a continuous version, still denoted by $\tilde{\mathbf{u}}$, such that
\[
\tilde{\mathbf{u}} \in C([0, T]; L^2(U)) \cap L^2(0, T; H^{1,2}_0(U)).
\]
We next show that $\tilde{\mathbf{u}}(0) = g$ in $L^2(U)$. By Theorem \ref{equiproweakso}, the identity \eqref{iterativeintlim} implies that for all $\mathbf{v} \in L^2(0,T; H^{1,2}_0(U))$,
\begin{equation} \label{weakformu}
\int_0^T \langle \tilde{\mathbf{u}}'(t), \mathbf{v}(t) \rangle_{H^{-1,2}(U)} \, dt 
+ \int_0^T \int_U \langle A \nabla \tilde{\mathbf{u}}(t), \nabla \mathbf{v}(t) \rangle 
+ \langle \mathbf{H}, \nabla \tilde{\mathbf{u}}(t) \rangle \mathbf{v}(t) 
+ (c + \theta) \tilde{\mathbf{u}}(t) \mathbf{v}(t) \, dx \, dt = 0.
\end{equation}
We now take $\tilde{\mathbf{v}} \in C^1([0, T]; H^{1,2}_0(U))$ with $\tilde{\mathbf{v}}(T) = 0$. Then, substituting $\tilde{\mathbf{v}}$ into \eqref{weakformu} and applying integration by parts in time, we obtain
\begin{align}
&-\int_0^T \langle \tilde{\mathbf{v}}'(t), \tilde{\mathbf{u}}(t) \rangle_{H^{-1,2}(U)} \, dt 
- \langle \tilde{\mathbf{v}}(0), \tilde{\mathbf{u}}(0) \rangle_{H^{-1,2}(U)} \nonumber \\
&\quad + \int_0^T \int_U \langle A \nabla \tilde{\mathbf{u}}(t), \nabla \tilde{\mathbf{v}}(t) \rangle 
+ \langle \mathbf{H}, \nabla \tilde{\mathbf{u}}(t) \rangle \tilde{\mathbf{v}}(t) 
+ (c + \theta) \tilde{\mathbf{u}}(t) \tilde{\mathbf{v}}(t) \, dx \, dt = 0. \label{weakformula2}
\end{align}
Similarly, applying Theorem \ref{equiproweakso} and integration by parts to \eqref{iterativeint2}, we obtain
\begin{align}
&-\int_0^T \langle \tilde{\mathbf{v}}'(t), \mathbf{u}_n(t) \rangle_{H^{-1,2}(U)} \, dt - \langle \tilde{\mathbf{v}}(0), g \rangle_{H^{-1,2}(U)} \nonumber \\
&\quad + \int_0^T \int_U \langle A_n \nabla \mathbf{u}_n(t), \nabla \tilde{\mathbf{v}}(t) \rangle 
+ \langle \mathbf{H}, \nabla \mathbf{u}_n(t) \rangle \tilde{\mathbf{v}}(t) 
+ (c + \theta) \mathbf{u}_n(t) \tilde{\mathbf{v}}(t) \, dx \, dt = 0. \label{weakformori}
\end{align}
Passing to the limit in \eqref{weakformori} by the weak convergence \eqref{weaklysobof},
we obtain
\begin{align}
&-\int_0^T \langle \tilde{\mathbf{v}}'(t), \tilde{\mathbf{u}}(t) \rangle_{H^{-1,2}(U)} \, dt - \langle \tilde{\mathbf{v}}(0), g \rangle_{H^{-1,2}(U)} \nonumber \\
&\quad + \int_0^T \int_U \langle A \nabla \tilde{\mathbf{u}}(t), \nabla \tilde{\mathbf{v}}(t) \rangle 
+ \langle \mathbf{H}, \nabla \tilde{\mathbf{u}}(t) \rangle \tilde{\mathbf{v}}(t) 
+ (c + \theta) \tilde{\mathbf{u}}(t) \tilde{\mathbf{v}}(t) \, dx \, dt = 0. \label{weakformori2}
\end{align}
Therefore, combining \eqref{weakformula2} and \eqref{weakformori2}, we conclude that
\[
\int_U (\tilde{\mathbf{u}}(0) - g) \tilde{\mathbf{v}}(0) \, dx = 0.
\]
Since $\tilde{\mathbf{v}}(0)$ can be arbitrary function in $H^{1,2}_0(U)$, we have $\tilde{\mathbf{u}}(0)=g$ in $L^2(U)$. Therefore, we conclude that $\tilde{\mathbf{u}}$ is a weak solution to \eqref{bvpara} by Theorem \ref{equiproweakso}. Ultimately, by the uniqueness result in Theorem \ref{theouniquene},  
\begin{equation*} %\label{uniqtiliroi}
\tilde{\mathbf{u}}=\mathbf{u}\quad\text{  in $C([0,T]; L^2(U))$}.
\end{equation*}
Therefore, we deduce from  \eqref{weaklysobof} that
%\begin{equation} \label{tildehatequ}
%\mathbf{u}(t) = \hat{\mathbf{u}}(t) \quad \text{ for a.e. $t \in (0,T)$},
%\end{equation}
\begin{equation} \label{fundaweaklim}
\lim_{n \rightarrow \infty}\mathbf{u}_n = \mathbf{u} \quad \text{ weakly in $L^2(0,T; H^{1,2}_0(U))$}.
\end{equation}
%\begin{equation} \label{fudweakliml2}
%\lim_{n \rightarrow \infty} \mathbf{u}_n(t)= \mathbf{u}(t) \quad \text{ weakly in $L^2(U)$} \quad \text{ for a.e. $t \in (0, \infty)$}.
%\end{equation}
\\ 
{\sf \underline{Step 2)} Transformation to an equation with weakly divergence-free vector fields}\\
Meanwhile, by Theorem \ref{existinvar}(i), (iii), there exists $\rho_n \in H^{1,2}(B_{4R}(x_0)) \cap C(B_{4R}(x_0))$ with $\rho_n \in H^{1,p}(U)$ and $\rho_n(x)>0$ for all $x \in B_{4R}(x_0)$ such that
    \begin{equation} \label{ellipticlee}
        \int_{B_{4R}(x_0)} \langle A_n^T \nabla \rho_n + \rho_n \mathbf{H}, \nabla \varphi \rangle dx = 0
        \quad \text{for all } \varphi \in C_0^\infty(B_{4R}(x_0)).
    \end{equation}
Moreover, by Theorem \ref{existinvar}(ii), there exists a constant $K_1 \geq 1$ which only depends on $d, \lambda, M, R, p, \|h\|_{L^p(U)}$ such that
    \begin{equation} \label{harnackineq}
        1 \leq \frac{\max_{\overline{U}} \rho_n}{\min_{\overline{U}} \rho_n} \leq K_1.
    \end{equation}
    Define $\mu_n=\rho_n dx$ and
\[
\mathbf{B}_n := \mathbf{H} + \frac{1}{\rho_n} A_n^T \nabla \rho_n \quad \text{on } U.
\]
Then, \eqref{ellipticlee} implies that
$$
\int_{U} \langle \mathbf{B}_n, \nabla \varphi \rangle d\mu_n=0 \quad \text{ for all $\varphi \in C_0^{\infty}(U)$}.
$$
Define a Dirichlet form $(\mathcal{E}_n, D(\mathcal{E}_n))$ as the closure of 
\begin{equation*} \label{underlydf2}
\mathcal{E}_n(f,g) = \int_U \langle A_n \nabla f, \nabla g \rangle \, d\mu_n 
+ \int_U \langle \mathbf{B}_n, \nabla f \rangle g \, d\mu_n + \int_U c f g \,d\mu_n, \quad f, g \in C_0^\infty(U),
\end{equation*}
(cf. \eqref{underlydf}). Let $(T_t^{(n)})_{t>0}$ be the sub-Markovian $C_0$-semigroup of contractions on $L^2(U,\mu_n)$ associated with $(\mathcal E_n,D(\mathcal E_n))$. Let
$\tilde{\mathbf{u}}_n(t):= e^{-\theta t} T^{(n)}_t g$, $t \in (0, \infty)$ and $\tilde{\mathbf{u}}_n(0):=g$ in $L^2(U)$.
By Theorem \ref{weaksolsemi}(i), (ii), $\tilde{\mathbf{u}}_n \in L^2(0,T;H^{1,2}_0(U)) \cap C([0,\infty); L^{2}(U))$ satisfies for any $v \in H^{1,2}_0(U)$ and $\eta \in C_0^{\infty}((0,T))$
\begin{equation} \label{tildeunsol}
\int_0^T \int_{U} \tilde{\mathbf{u}}_n(t) \partial_t (v \eta) \, d\mu_n dt= \int_0^T \int_U \left( \langle A_n \nabla \tilde{\mathbf{u}}_n, \nabla (v \eta) \rangle 
+ \langle \mathbf{B}_n, \nabla \tilde{\mathbf{u}}_n \rangle (v \eta) + (c+\theta) \tilde{\mathbf{u}}_n (v \eta) \,d\mu_n \right) \,dt,
\end{equation}
and
\begin{equation*} %\label{energyrhon}
\int_0^T \int_{U}  \|\nabla \tilde{\mathbf{u}}_n(t) \|^2 d\mu_n dt  \leq \frac{1}{2\lambda} \| g\|^2_{L^2(U, \mu_n)}.
\end{equation*}
Thus, by \eqref{harnackineq},
\begin{equation} \label{energyrhonnew}
\|\tilde{\mathbf{u}}_n \|_{L^2(0,T;H^{1,2}_0(U))}  \leq \left(\frac{K_1}{2\lambda} \right)^{\frac12} \| g\|_{L^2(U)}.
\end{equation}
Note that it follows from Theorem \ref{leetransform} with $\rho$ and $\mathbf{B}$ replaced by $\rho_n$ and $\mathbf{B}_n$, respectively, that $\tilde{\mathbf{u}}_n \in L^2(0,T; H^{1,2}_0(U)) \cap C([0,T]; L^{2}(U))$ satisfies $\tilde{\mathbf{u}}_n(0) = g$ in $L^2(U)$ and
\begin{align*} \label{leetransfores}
&\int_0^T \int_U \tilde{\mathbf{u}}_n(t) \, \partial_t(\psi \eta) \, dx \, dt \\
&= \int_0^T \int_U \langle A_n \nabla \tilde{\mathbf{u}}_n(t), \nabla (\psi \eta) \rangle 
+ \langle \mathbf{H}, \nabla \tilde{\mathbf{u}}_n(t) \rangle (\psi \eta) 
+(c+\theta) \tilde{\mathbf{u}}_n(t) (\psi \eta) \, dx \, dt \quad \text{ for all $\psi \in H^{1,2}_0(U)$ and $\eta \in C_0^\infty((0,T))$. }
\end{align*}
By the uniqueness result in Theorem \ref{theouniquene}, 
\begin{equation}  \label{uniqueuptoinf}
\mathbf{u}_n = \tilde{\mathbf{u}}_n \quad \text{in } C([0,T]; L^2(U)) \quad \text{ for all $n \geq 1$}.
\end{equation}
%Indeed, since $T>0$ is arbitrarily chosen, 
%\begin{equation} \label{uniqueuptoinf}
%\mathbf{u}_n = \tilde{\mathbf{u}}_n \quad \text{in } C([0,\infty); L^2(U)).
%\end{equation}
Finally, the assertion follows from \eqref{energyrhonnew} and \eqref{fundaweaklim}\\ \\
(ii) Applying Theorem \ref{weaksolsemi}(ii) to \eqref{tildeunsol},
\begin{equation} \label{l2contrarhonnew}
\|\tilde{\mathbf{u}}_n(t)\|_{L^2(U, \mu_n)} \leq  e^{-\kappa_{0,n}t}\|g\|_{L^2(U, \mu_n)} \quad \text{ for all $t \in [0, T]$},
\end{equation}
where 
$$
\kappa_{0,n}=\theta+ \frac{d^2 \lambda }{4(d-1)^2 |U|^{\frac{2}{d}}} \left(\frac{\inf_{U} \rho_n}{\sup_{U} \rho_n} \right).
$$
By Theorem \ref{existinvar}(ii) and \eqref{harnackineq}, $\kappa_{0,n} \geq \kappa$, and hence it follows from \eqref{l2contrarhonnew}, \eqref{harnackineq} and \eqref{uniqueuptoinf} that
\begin{equation*} 
\|\mathbf{u}_n(t)\|_{L^2(U)} \leq K_1^{\frac12}e^{-\kappa t} \|g\|_{L^2(U)} \quad \text{ for all $t \in [0, T]$}.
\end{equation*}
By Lemma \ref{weckcol2lem} and \eqref{fundaweaklim},
\begin{equation*} %\label{l2controrigina}
\|\mathbf{u}(t)\|_{L^2(U)} \leq K_1^{\frac12}e^{-\kappa t} \|g\|_{L^2(U)} \quad \text{ for a.e. $t \in (0, T)$}.
\end{equation*}
Since $\mathbf{u} \in C([0, \infty); L^2(U))$ and $T \in (0, \infty)$ is arbitrarily chosen, \eqref{expdecaacer} follows. \\ \\
(iii) 
\underline{\textsf{Step 1:}} Apply the Banach–Saks theorem (cf. \cite[Appendix A, Theorem 2.2]{MR92}) to the sequence $(\mathbf{u}_n)_{n \geq 1}$ in \eqref{fundaweaklim}. Then there exists a subsequence, still denoted by $(\mathbf{u}_n)_{n \geq 1}$, such that the Cesàro means
\begin{equation*} %\label{subsequen}
\mathbf{s}_N := \frac{1}{N} \sum_{k=1}^N \mathbf{u}_k, \quad N \geq 1,
\end{equation*}
converge strongly to $\mathbf{u}$ in $L^2(0,T; H_0^{1,2}(U))$.
Hence, there exists a strictly increasing sequence $(n_k)_{k \geq 1} \subset \mathbb{N}$ such that
\begin{equation}\label{subsequnk}
\lim_{k \to \infty} \mathbf{s}_{n_k}(t) = \mathbf{u}(t) \quad \text{in } H_0^{1,2}(U) \quad \text{for a.e. } t \in (0,T).
\end{equation}
Now define the set
\begin{equation} \label{defnofjset}
J := \left\{ t \in (0,T) : \lim_{k \to \infty} \mathbf{s}_{n_k}(t) = \mathbf{u}(t) \quad \text{ in $H^{1,2}_0(U)$} \right\}.
\end{equation}
Then, we have $| (0,T) \setminus J | = 0$. \\ \\
\underline{\sf Step 2:} Let $g \in L^{\infty}(U)$ and fix $t_1 \in J$. 
We will show \eqref{linftycontra}. 
Applying Theorem \ref{weaksolsemi}(iii) to \eqref{tildeunsol} and using \eqref{uniqueuptoinf}, we have
\[
\| \mathbf{u}_n(t_1) \|_{L^{\infty}(U)} \leq e^{-\theta t_1} \|g\|_{L^{\infty}(U)} \quad \text{for all } n \geq 1,
\]
and hence by the triangle inequality,
\begin{equation} \label{cesarolinfbd}
\| \mathbf{s}_{n_k}(t_1) \|_{L^{\infty}(U)} \leq e^{-\theta t_1} \|g\|_{L^{\infty}(U)} \quad \text{for all } k \geq 1,
\end{equation}
where $(n_k)_{k \geq 1} \subset \mathbb{N}$ is a sequence as in \eqref{subsequnk}. By the Banach--Alaoglu theorem, there exists a subsequence of $(n_k)_{k \geq 1}$, denoted again by $(n_k)_{k \geq 1}$ and a function $\mathbf{u}^*(t_1) \in L^{\infty}(U)$ such that
\[
\lim_{k \to \infty} \mathbf{s}_{n_k}(t_1) = \mathbf{u}^*(t_1) \quad \text{weakly$^*$ in } L^{\infty}(U),
\]
i.e., for all $\varphi \in L^1(U)$,
\[
\lim_{k \to \infty} \int_U \mathbf{s}_{n_k}(t_1) \, \varphi \, dx = \int_U \mathbf{u}^*(t_1) \, \varphi \, dx.
\]
Since it follows from the definition in \eqref{defnofjset} that $\lim_{k \rightarrow \infty}\mathbf{s}_{n_k}(t_1)=\mathbf{u}(t_1)$ in $H^{1,2}_0(U)$, we have
$$
\mathbf{u}^*(t_1)=\mathbf{u}(t_1)\quad \text{ in $L^2(U)$}.
$$
Since $\mathbf{s}_{n_k}(t_1) \to \mathbf{u}(t_1)$ weakly$^*$ in $L^{\infty}(U)$ as $k \rightarrow \infty$, by \eqref{cesarolinfbd}, the weak$^*$ lower semicontinuity of the $L^{\infty}$-norm yields
\begin{equation*} %\label{uniquetilde}
\|\mathbf{u}(t_1)\|_{L^{\infty}(U)} \leq \liminf_{k \to \infty} \|\mathbf{s}_{n_k}(t_1)\|_{L^{\infty}(U)} \leq e^{-\theta t_1} \|g\|_{L^{\infty}(U)}.
\end{equation*}
Since $t_1 \in J$ is an arbitrary point,
\begin{equation} \label{arbtjbound}
\|\mathbf{u}(t)\|_{L^{\infty}(U)} \leq e^{-\theta t} \|g\|_{L^{\infty}(U)} \quad \text{for all } t \in J.
\end{equation}
Let $t_* \in (0,T) \setminus J$ be arbitrarily fixed. Since $| (0,T) \setminus J | = 0$, we can extract a sequence of $(t_n)_{n \geq 1}$ in $J$ such that $\lim_{n \rightarrow \infty}t_n=t_*$. Thus, $\mathbf{u} \in C([0, \infty); L^2(U))$ implies that
\begin{equation*} 
\lim_{n \to \infty} \mathbf{u}(t_n) = \mathbf{u}(t_*) \quad \text{ in $L^2(U)$}.
\end{equation*} 
Thus, there exists a subsequence of $(t_n)_{n \geq 1} \subset J$, say again $(t_n)_{n \geq 1}$ such that
\begin{equation} \label{convergenc}
\lim_{n \to \infty} \mathbf{u}(t_n) = \mathbf{u}(t_*) \quad \text{ for a.e. $x \in U$}.
\end{equation}
Applying \eqref{arbtjbound} to each $t_n$ and using \eqref{convergenc}, we have
\begin{equation*} 
\|\mathbf{u}(t_*)\|_{L^{\infty}(U)} \leq e^{-\theta t_*} \|g\|_{L^{\infty}(U)}.
\end{equation*}
Since $t_*$ was arbitrary in $(0, T) \setminus J$,  it holds that
\[
\|\mathbf{u}(t)\|_{L^{\infty}(U)} \leq e^{-\theta t} \|g\|_{L^{\infty}(U)} \quad \text{for all } t \in (0,T).
\]
Since $T>0$ can be chosen arbitrarily, \eqref{linftycontra} follows.
\text{}\\ \\
\underline{\sf Step 3:} Let $r \in [2, \infty)$, $g \in L^r(U)$ and fix $t_1 \in J$. We aim to prove \eqref{linftcontrarca}.  
Applying Theorem \ref{weaksolsemi}(iii) to \eqref{tildeunsol} and using \eqref{harnackineq} and \eqref{uniqueuptoinf}, we obtain
\begin{equation} \label{imptestimr}
\| \mathbf{u}_n(t_1) \|_{L^{r}(U)} \leq K_1^{\frac1r}e^{-\theta t_1} \|g \|_{L^{r}(U)} \quad \text{for all } n \geq 1,
\end{equation}
and hence by the triangle inequality,
\begin{equation} \label{cestrianglein}
\| \mathbf{s}_{n_k}(t_1) \|_{L^{r}(U)} \leq K_1^{\frac1r}e^{-\theta t_1} \|g\|_{L^{r}(U)} \quad \text{for all } k \geq 1,
\end{equation}
where $(n_k)_{k \geq 1}\subset \mathbb{N}$ is a sequence as in \eqref{subsequnk}. 
By the weak compactness of $L^r(U)$, there exists a subsequence of $(n_k)_{k \geq 1}$, denoted again by $(n_k)_{k \geq 1}$ and a function $\mathbf{u}^*(t_1) \in L^{r}(U)$ such that
\[
\lim_{k \to \infty} \mathbf{s}_{n_k}(t_1) = \mathbf{u}^*(t_1) \quad \text{weakly in } L^r(U).
\]
In particular, by the definition in \eqref{defnofjset}, $\mathbf{u}^*(t_1) = \mathbf{u}(t_1)$. Moreover, we have
\[
\|\mathbf{u}(t_1)\|_{L^r(U)} \leq \liminf_{k \to \infty} \| \mathbf{s}_{n_k}(t_1) \|_{L^r(U)} \leq K_1^{\frac1r}e^{-\theta t_1} \|g\|_{L^r(U)},
\]
Finally, by the similar argument as in {\sf Step 2} and applying Fatou's lemma, we conclude \eqref{linftcontrarca}.
\text{}\\ \\
\underline{\sf Step 4:} Let $r \in [1, 2)$, $g \in L^2(U)$ and fix $t_1 \in J$. We will also prove \eqref{linftcontrarca}.  
Applying Theorem \ref{weaksolsemi}(iii) to \eqref{tildeunsol} and using \eqref{harnackineq} and \eqref{uniqueuptoinf}, we have
\eqref{imptestimr}
and hence by the triangle inequality, we get \eqref{cestrianglein}, where $(n_k)_{k \geq 1}$ is a sequence as in \eqref{subsequnk}. 
Since
\[
\lim_{k \to \infty} \mathbf{s}_{n_k}(t_1) = \mathbf{u}(t_1) \quad \text{ in } H^{1,2}_0(U)\; \text{ (and hence weakly in } L^r(U)),
\]
we obtain
\[
\|\mathbf{u}(t_1)\|_{L^r(U)} \leq \liminf_{k \to \infty} \|\mathbf{s}_{n_k}(t_1)\|_{L^r(U)} \leq K_1^{\frac1r }e^{-\theta t_1} \|g\|_{L^r(U)}.
\]
By the same reasoning as in {\sf Step 2} and applying Fatou's lemma, we conclude \eqref{linftcontrarca}.
\text{}\\ \\
\underline{\sf Step 5:} We will show that $\mathbf{u} \in C([0, \infty); L^r(U))$ for each $r \in [1, \infty)$. 
Since $\mathbf{u} \in C([0, \infty); L^2(U))$, it directly follows that $\mathbf{u} \in C([0, \infty); L^r(U))$ for each $r \in [1, 2]$.
Now fix $r \in (2, \infty)$ and let $g \in L^r(U)$. Then, there exists a sequence of functions $(g_n)_{n \geq 1} \subset L^{\infty}(U)$ such that $g_n \to g$ in $L^r(U)$. 
For each $\phi \in L^2(U)$, denote by $\tilde{\mathbf{u}}^{\phi} \in C([0, \infty); L^2(U)) \cap L^2(0,T; H^{1,2}_0(U))$ the unique weak solution to \eqref{bvpara} with initial condition $\phi$, as constructed in Theorem~\ref{mainwellposth}. 
Then, by linearity and uniqueness, we have for any $\phi, \psi \in L^2(U)$,
\[
\tilde{\mathbf{u}}^{\phi} -\tilde{\mathbf{u}}^{\psi} =\tilde{\mathbf{u}}^{\phi-\psi} \quad \text{in } C([0, \infty); L^2(U)).
\]
%In particular, Theorem \ref{maintheore}(ii) implies
%\begin{equation} \label{l2contrapr}
%\| \tilde{\mathbf{u}}^{\phi}(t) -\tilde{\mathbf{u}}^{\psi}(t) \|_{L^2(U)} \leq  K_1^{1/2} e^{-\kappa t}  \| \phi - \psi \|_{L^2(U)} \quad %\text{for all } t \in [0, \infty).
%\end{equation}
Moreover, by \eqref{linftcontrarca} in {\sf Step 3}, for all $\phi, \psi \in L^r(U)$,
\begin{equation} \label{lrcontrapr}
\| \tilde{\mathbf{u}}^{\phi}(t) -\tilde{\mathbf{u}}^{\psi}(t) \|_{L^r(U)} \leq  K_1^{\frac{1}{r}} e^{-\theta t}  \| \phi - \psi \|_{L^r(U)}.
\end{equation}
Let $f \in L^{\infty}(U)$. Then we claim $\tilde{\mathbf{u}}^{f} \in C([0, \infty); L^r(U))$. Indeed, for fixed $t_0 \in [0, \infty)$, we estimate using \eqref{linftycontra} in {\sf Step 2} and  $\tilde{\mathbf{u}}^{f} \in C([0, \infty); L^2(U))$:
\begin{align}
\|\tilde{\mathbf{u}}^{f}(t) - \tilde{\mathbf{u}}^{f}(t_0)\|_{L^r(U)}^r &= \int_U | \tilde{\mathbf{u}}^{f}(t) - \tilde{\mathbf{u}}^{f}(t_0) |^r dx \notag \\
&\leq \left( \|\tilde{\mathbf{u}}^{f}(t)\|_{L^{\infty}(U)} + \|\tilde{\mathbf{u}}^{f}(t_0)\|_{L^{\infty}(U)} \right)^{r-2} \| \tilde{\mathbf{u}}^{f}(t) - \tilde{\mathbf{u}}^{f}(t_0) \|_{L^2(U)}^2 \notag \\
&\leq \left( (e^{-\theta t} + e^{-\theta t_0})  \|f\|_{L^{\infty}(U)}\right)^{r-2} \| \tilde{\mathbf{u}}^{f}(t) - \tilde{\mathbf{u}}^{f}(t_0) \|^2_{L^2(U)} \to 0 \quad \text{as } t \to t_0, \label{strcontifpr}
\end{align}
and hence the claim is shown.\\
Finally, for general $g \in L^r(U)$, using the triangle inequality and \eqref{lrcontrapr}, we estimate: for each $n \geq 1$
\begin{align*}
\| \tilde{\mathbf{u}}^{g}(t) - \tilde{\mathbf{u}}^{g}(t_0) \|_{L^r(U)} &\leq \| \tilde{\mathbf{u}}^{g}(t) - \tilde{\mathbf{u}}^{g_n}(t) \|_{L^r(U)} + \| \tilde{\mathbf{u}}^{g_n}(t) - \tilde{\mathbf{u}}^{g_n}(t_0) \|_{L^r(U)} \\
&\quad + \| \tilde{\mathbf{u}}^{g_n}(t_0) - \tilde{\mathbf{u}}^{g}(t_0) \|_{L^r(U)} \\
&\leq K_1^{\frac{1}{r}} e^{-\theta t} \|g - g_n\|_{L^r(U)} + \| \tilde{\mathbf{u}}^{g_n}(t) - \tilde{\mathbf{u}}^{g_n}(t_0) \|_{L^r(U)} + K_1^{\frac{1}{r}} e^{-\theta t_0} \|g - g_n\|_{L^r(U)}.
\end{align*}
Combining the above with \eqref{strcontifpr} and applying a standard $3\varepsilon$-argument, 
$
\lim_{t \rightarrow t_0}\| \tilde{\mathbf{u}}^{g}(t) - \tilde{\mathbf{u}}^{g}(t_0) \|_{L^r(U)} =0.
$
Hence,
\[
\mathbf{u} \in C([0, \infty); L^r(U)) \quad \text{for all } r \in [1, \infty).
\]
\hfill$\square$

\section{Discussion} \label{discuss}

This paper has established not only the existence and uniqueness of solutions to the parabolic initial-boundary value problem \eqref{bvpara}, but also the exponential $L^2$-stability of the solution as $t \to \infty$. The exponential stability robustly holds under the assumption that the drift vector field $\mathbf{H}$ belongs to $L^p(U, \mathbb{R}^d)$ for some $p \in (d, \infty)$, and that the zero-order coefficient satisfies $c + \theta \in L^s(U)$ with $s \in (1, \infty)$ if $d = 2$ and $s = \frac{d}{2}$ if $d \ge 3$.
A natural and challenging direction for future investigation is whether the exponential $L^2$-stability remains valid under the more general condition \textbf{(Hy)} in Section~\ref{semigrouponl2dx}. The core of the current method lies in constructing a weight $\mu = \rho dx$ (Theorem~\ref{existinvar}) under the structural condition \textbf{(S)}, which is crucial for applying the divergence-free transformation (Theorem~\ref{leetransform}) and deriving a constant $K_1 \geq 1$ via an elliptic Harnack inequality. Since Theorem~\ref{existinvar} heavily relies on \textbf{(S)}, extending the result to the broader hypothesis on the drift coefficients as in \textbf{(Hy)} would require fundamentally new ideas, and constitutes a significant open problem. \\
Another important direction is to examine the pointwise exponential decay of the solution. More specifically, for each fixed point $x \in U$, and a weak solution $u$ to \eqref{bvpara}, one may ask whether the modulus $|u(x,t)|$ decays exponentially as $t \to \infty$. Addressing this question would require establishing the continuity of the solution and deriving suitable pointwise estimates. These estimates are intimately connected with the decay properties of the stochastic representation, in the special case $c=\theta=0$,
\[
\mathbb{E}_x\left[g(X_t^x)\mathbf{1}_{\{t<\tau_U^x\}}\right],
\]
as introduced in \eqref{stochobjects}, thereby offering a natural interface between the analytic theory of PDEs and stochastic processes.
A further perspective involves connecting the long-time behavior of the solution to the invariant measure of the associated diffusion process $(X_t^x)_{t \ge 0}$.
 In many settings, there exists an invariant probability measure $\mu$ such that, under suitable ergodicity conditions and for suitable $g$,
\[
\lim_{t \to \infty} \mathbb{E}[g(X_t^x)] = \int_{\mathbb{R}^d} g\,d\mu,
\]
(see \cite[Section 4.7]{FOT11} and \cite[Theorem 5.2.26]{BKRS15}). In regular settings, $\mathbb{E}_x[g(X_t^x)]$ coincides with $T(t)g(x)$, where $(T(t))_{t \ge 0}$ is the semigroup associated with the operator (see \cite[Theorem 2.5.2]{Lo07}). Thus, exponential convergence of the semigroup is closely connected to exponential ergodicity in the probabilistic sense, and in many classical settings these properties are governed by the spectral gap of the corresponding generator (cf. \cite{HHS05, BGL14}).
The coercivity of the drift term is known to ensure exponential convergence of the associated stochastic dynamics, and this property has been extended to various classes of nonlinear stochastic differential equations in \cite{EGZ19, BRS18}. Furthermore, exponential convergence has also been established even in the presence of significant degeneracy in the diffusion coefficients, as in the case of underdamped Langevin dynamics (see \cite{GW19}).
In connection with this, \cite[Theorem 1.1(iii)]{L25ai} establishes that the one-dimensional marginal distributions of solutions to SDEs, which are analytically characterized by the corresponding Fokker–Planck equations, converge to invariant measures under suitable conditions on the drift coefficient. Furthermore, without relying on any spectral gap assumption, \cite[Theorem 1.1]{Por24} demonstrates the exponential convergence of solutions in the $L^1$ norm with respect to an associated weighted measure.
\\ 
Hence, establishing exponential convergence to zero or to an invariant measure under minimal assumptions such as $\mathbf{H} \in L^p_{\text{loc}}(\mathbb{R}^d, \mathbb{R}^d)$ remains a deep and difficult problem. For instance, on the whole space $\mathbb R^d$, even the Brownian heat semigroup exhibits only polynomial $L^1$-to-$L^\infty$ decay of order $t^{-d/2}$, rather than exponential decay. Thus, a complete characterization of the convergence rate in such general settings would be a highly significant and nontrivial contribution. \\
Finally, although theoretical results suggest that drift coefficients can accelerate convergence (see \cite{HHS05}), this effect was not explicitly observed in our analysis. While our results demonstrate the robustness of exponential decay even in the presence of nonsymmetric drifts, whether such drifts quantitatively accelerate convergence remains an open question. Investigating this issue in the context of invariant measures could further clarify the convergence mechanisms of drift-driven dynamics.  Such insights are closely related to the theory underlying Markov chain Monte Carlo (MCMC) algorithms (cf. \cite{MCF15}), which rely on convergence to an invariant distribution and are widely used in applied fields such as image generation as in \cite{S19, SS21}. Developing rigorous analytic foundations for such phenomena remains an important and challenging goal.

\newpage

\section*{\normalsize Acknowledgments}
This work was supported by the National Research Foundation of Korea (NRF) grant funded by the Korea government (MSIT) (RS-2026-25598810) and the Gyeongbuk ANCHOR system-(Regional Growth Innovation LAB) program through the Gyeongbuk ANCHOR Center, funded by the Ministry of Education (MOE) and the Gyeongsangbuk-do, Republic of Korea (2026-anchor-15-105).

\section*{\normalsize Author contributions}
This manuscript is a single-author paper, and the author completed all aspects of the work.

\section*{\normalsize  Data availability}
No datasets were generated or analysed during the current study.

\section*{\normalsize  Ethics approval and consent to participate}
Not applicable.

\section*{\normalsize  Consent for publication}
Not applicable.

\section*{\normalsize  Competing interests}
The author declares no competing interests.

\centerline{}
\centerline{}
Haesung Lee\\
Department of Mathematics and Big Data Science,  \\
Kumoh National Institute of Technology, \\
Gumi, Gyeongsangbuk-do 39177, Republic of Korea, \\
E-mail: fthslt@kumoh.ac.kr, \; fthslt14@gmail.com
\end{document}